\documentclass[a4paper,12pt]{amsart}

\usepackage{amsmath, amssymb}
\usepackage{tikz}

\allowdisplaybreaks

\newtheorem{thm}{Theorem}[section]
\newtheorem{cor}[thm]{Corollary}
\newtheorem{lem}[thm]{Lemma}
\newtheorem{prop}[thm]{Proposition}

\newtheorem*{lem*}{Lemma}

\theoremstyle{definition}

\theoremstyle{remark}
\newtheorem{rmk}[thm]{Remark}

\newtheorem{example}[thm]{Example}

\numberwithin{equation}{section}

\newcommand{\NN}{\mathbb{N}}

\newcommand{\RR}{\mathbb{R}}
\newcommand{\TT}{\mathbb{T}}
\newcommand{\CC}{\mathbb{C}}

\def\Cc{\mathcal{C}}
\def\Ccrit{\mathcal{C}_{\operatorname{crit}}}
\def\Cmcrit{\mathcal{C}_{\operatorname{mincrit}}}

\newcommand{\KMS}{\text{KMS}}
\newcommand{\TC}{\mathcal{T}C}
\newcommand{\Aut}{\operatorname{Aut}}
\newcommand{\QE}{\operatorname{QE}}
\newcommand{\clsp}{\overline{\text{span}}}
\newcommand{\lspan}{\operatorname{span}}

\newcommand\mc{\mathrm{mc}}

\date{January 2, 2018}

\begin{document}

\author[James Fletcher]{James Fletcher}
\author[Astrid an Huef]{Astrid an Huef}
\author[Iain Raeburn]{Iain Raeburn}
\email{james.fletcher, astrid.anhuef, iain.raeburn@vuw.ac.nz}
\address{School of Mathematics and Statistics, Victoria University of Wellington, P.O. Box 600, Wellington 6140, New Zealand.}

\title[KMS states on higher-rank graph algebras]{A program for finding all KMS states\\on the Toeplitz algebra of a higher-rank graph}

\thanks{This research was supported by the Marsden Fund of the Royal Society of New Zealand.}

\begin{abstract}
The Toeplitz algebra of a finite graph of rank $k$ carries a natural action of the torus $\TT^k$, and composing with an embedding of $\RR$ in $\TT^k$ gives a dynamics on the Toeplitz algebra. For inverse temperatures larger than a critical value, the KMS states for this dynamics are well-understood, and this analysis is the first step in our program. At the critical inverse temperature, much less is known, and the second step in our program is an analysis of the KMS states at the critical value. This is the main technical     contribution of the present paper.  The third step shows that the problem of finding the states at inverse temperatures less than the critical value is equivalent to our original problem for a smaller graph. Then we can tackle this new problem using the same three steps, and repeat if necessary. So in principle, modulo some mild connectivity conditions on the graph, our results give a complete description of the simplex of KMS states at all inverse temperatures. We test our program on a wide range of examples, including a very general family of graphs with three strongly connected components.
\end{abstract}
\maketitle

\section{Introduction}

The graphs of higher-rank $k$ (the $k$-graphs) were introduced by Kumjian and Pask \cite{KP} as combinatorial models for the higher-rank Cuntz--Krieger algebras of Robertson and Steger \cite{RobS}. To each $k$-graph $\Lambda$ we associate two $C^*$-algebras: a graph $C^*$-algebra $C^*(\Lambda)$ \cite{KP} and a Toeplitz algebra $\TC^*(\Lambda)$ \cite{RS}. Both algebras carry natural gauge actions of the torus $\TT^k$. Composing these gauge actions with an embedding $t\mapsto e^{itr}$ of $\RR$ in $\TT^k$ gives actions $\alpha^r$ of $\RR$ on $C^*(\Lambda)$ and $\TC^*(\Lambda)$ (the dynamics). Then one naturally wonders about the KMS states of these dynamics, and, as usual, the results turn out to be interesting \cite{aHLRS2, aHLRS3, aHKR, aHKR2}.

We assume throughout that $\Lambda$ is a finite $k$-graph with no sources or sinks. For inverse temperatures $\beta$ larger than a critical value $\beta_c$, the system $(\TC^*(\Lambda),\RR,\alpha^r)$ has a concretely described simplex of KMS$_\beta$ states with extreme points parametrised by the vertices in the graph \cite[Theorem~6.1]{aHLRS2}. The critical value $\beta_c$ is determined by the spectral radii $\{\rho(A_i):1\leq i\leq k\}$ of the vertex matrices $\{A_i:1\leq i\leq k\}$ of $\Lambda$ as
\[
\beta_c=\max\{r_i^{-1}\ln\rho(A_i):1\leq i\leq k\}.
\]
For the ``preferred dynamics'' in which $r_i=\ln\rho(A_i)$ for all $i$, the critical inverse temperature is $\beta_c=1$. 

Roughly speaking, when $C^*(\Lambda)$ is simple and the dynamics is not periodic,  there is a unique KMS$_{1}$ state on $\TC^*(\Lambda)$  (\cite[Theorem~7.2]{aHLRS2}, \cite[Theorem~5.1(d)]{aHKR});  when the dynamics is preferred  it factors through a KMS$_1$ state of $C^*(\Lambda)$. 

This interesting topic has been picked up by a number of authors from different points of view. In particular, Yang \cite{Yang3} and Laca--Larsen--Neshveyev--Sims--Webster \cite{LLNSW} have computed the types of the von Neumann algebras associated to these KMS states, finding that the type depends only on the skeleton of the $k$-graph. Farsi, Gillaspy, Kang and Packer \cite{FGKP} have investigated spatial realisations of these KMS states, and found connections with constructions of wavelets. McNamara \cite{M} has extended the results of \cite{aHLRS2} to more general dynamics. The shifts on the infinite path spaces of higher-rank graphs provided examples of $*$-commuting local homeomorphisms which informed our analysis with Afsar \cite{AaHR2}. And, very recently, the uniqueness of the KMS$_1$ state for the preferred dynamics on the $C^*$-algebra of a simple graph has been confirmed by Christensen \cite{C} using work of Neshveyev  \cite{N} on KMS states of groupoid $C^*$-algebras.

Simplicity of a graph algebra $C^*(\Lambda)$ is largely determined by two properties: $\Lambda$ should be aperiodic and  irreducible  \cite{Robertson-Sims}. In \cite{aHLRS3}, we studied the KMS states of $C^*(\Lambda)$ for periodic irreducible graphs, and we found that the behaviour of the KMS states at the critical inverse temperature $1$ is dictated in a very concrete way by the periodicity \cite[Theorem~7.1]{aHLRS3}.  We recently investigated graphs with several irreducible components which are themselves aperiodic \cite{aHLRS4, aHKR2}. We were pleasantly surprised that we were able to get rather complete descriptions of the KMS states for $1$-graphs \cite[Theorem~5.3]{aHLRS4}, and that we were then able to extend many of the key arguments of \cite{aHLRS4} to higher-rank graphs. Our general results were strong enough to describe all the KMS$_1$ states for the preferred dynamics on graphs with one or two components \cite[\S7]{aHKR2}.

Unfortunately, even for the preferred dynamics, we ran into new difficulties  when the graph had three components. The program of \cite{aHKR2} involves reducing problems to smaller graphs by describing what happens when we remove hereditary components. The Toeplitz algebras of these smaller graphs are quotients of the Toeplitz algebra $\TC^*(\Lambda)$. However, the dynamics on the quotient induced by the preferred dynamics on $\TC^*(\Lambda)$ need not be the preferred one. For graphs with two components, the quotient is irreducible, and the analysis of \cite{aHKR} was available. For graphs with three components, the quotient graph could have two components. So we were forced to rejig our general results to accommodate a non-preferred dynamics, and that is what we report on here. We always scale our dynamics to ensure that $\beta_c=1$, and we are interested primarily in the KMS$_1$ states. And we can then bootstrap our techniques to deal with $\beta<1$.

Our approach has been to develop a collection of results which allow us to reduce the number of components. Some of these results were already available in some form in our previous articles, and for them our new contribution has been to finesse the arguments to allow for a non-preferred dynamics. (This is the case, for example, for our new Proposition~\ref{estnonpref} and Theorem~\ref{dominant2cpt}.) But some of our ingredients are quite different, and we have been very pleasantly surprised at how well they work. For example, Theorem~\ref{orderrho} and the results in \S\ref{partial} have no counterparts in our previous papers. 

These results are intended to be used in tandem, and in \S\ref{procedure} we describe a general procedure for computing all the KMS$_1$ states of a very large class of higher-rank graphs. We do have to place some additional hypotheses on our graphs, but these are mainly to ensure that the components themselves are tractable, and that we do not create sources when we remove components. (We saw in \cite[\S8]{aHKR2} that pretty strange things can happen when we remove a component.) We then devise a program that finds all KMS states on $\TC^*(\Lambda)$ at all inverse temperatures.  We are delighted to report that we have not hit any new roadblocks to our program. We have in particular tested our program for a non-preferred dynamics on the Toeplitz algebras of graphs with three components, with very satisfactory results (see \S\ref{sec3cpts} and \S\ref{examples}).

\subsection*{Outline}

We begin in \S\ref{background} by summarising the necessary background material on higher-rank graphs, their associated $C^*$-algebras, and $\KMS$ states. We then set about adapting the main general results from \cite{aHKR2}. In \S\ref{critcpts} we show how to identify and remove redundant components from our graphs. The key point is that, provided the graph satisfies some mild connectivity hypotheses, we may replace our original graph with a smaller graph in which every critical component is hereditary, and not lose any $\KMS_1$ states in the process.

We begin \S\ref{extendPF} by showing how we may extend the Perron--Frobenius eigenvector for the vertex matrices of a hereditary component of the graph to an eigenvector of the full vertex matrix. This result is linear-algebraic in nature, and its proof involves recasting the proof of  \cite[Proposition~6.1]{aHKR2} to minimise the use of (and implicitly the existence of) inverses of matrices. But we were very surprised to discover Theorem~\ref{orderrho} and Corollary~\ref{orderrho2}, which show (again subject to a mild connectivity hypothesis) how the presence of a hereditary component which is critical for one colour influences the ordering of the spectral radii for the other colours.  However, the proof of Theorem~\ref{orderrho} was sufficiently indirect to make us nervous, and in the end we worked hard to find a direct proof for graphs with only 3 vertices, which can be found in Appendix~\ref{app3vertex}. 

In \S\ref{sec:dom} we extend the main result \cite[Theorem~6.5]{aHKR2} from our previous paper to the situation where the dynamics are non-preferred. When $\Lambda$ has a critical hereditary component $D$ that dominates all other components --- in the sense that all the spectral radii of $\Lambda_D$ dominate the spectral radii of every other component  --- Theorem~\ref{dominant2cpt} completely describes the $\KMS_1$ states of $\TC^*(\Lambda)$. This result enables us to deal with graphs that have one or two components (with potentially non-preferred dynamics), but cannot be used if the graph has two or more critical components. 

In \S\ref{partial} we prove a completely new theorem which allows us to handle graphs with several components that are critical and hereditary. In \S\ref{procedure} we combine all of our results to provide a procedure for finding all of the $\KMS_1$ states on the Toeplitz algebra of a reducible higher-rank graph. In \S\ref{program} we describe our program for finding the KMS$_\beta$ states at all inverse temperatures $\beta$. In \S\ref{apps} we test our program on  graphs with two components (generalising \cite{aHKR2}, where the dynamics was the preferred one) and three components. Then we check that for $1$-graphs our results are compatible with what we already know for directed graphs \cite{aHLRS4}, and  we explicitly compute the KMS states for the Toeplitz algebras of two specific dumbbell $2$-graphs with three single-vertex components.

\section{Background}\label{background}

\subsection{Higher-rank graphs}
A higher-rank graph of rank $k$ (or $k$-graph) consists of a countable category $\Lambda$ and a degree functor $d:\Lambda\rightarrow \NN^k$ satisfying the following factorisation property: if $\lambda\in \Lambda$ and $d(\lambda)=m+n$ for some $m,n\in \NN^k$, then there are unique $\mu,\nu\in \Lambda$ such that $d(\mu)=m$, $d(\nu)=n$ and $\lambda=\mu\nu$.  

We use the following standard notation when working with $k$-graphs. For $n\in \NN^k$, we write $\Lambda^n:=d^{-1}(n)$, and call elements of $\Lambda^n$ paths of degree $n$; in particular, elements of $\Lambda^0$ are called vertices. The factorisation property implies that we can identify the vertices  with the objects in the category. We write $s,r:\Lambda\rightarrow \Lambda^0$ for the domain and codomain maps. For $E\subset \Lambda$ and $V\subset\Lambda^0$ we set $VE:=\{\lambda\in E: r(\lambda)\in V\}$ and $vE:=\{v\}E$, and similarly on the right. We write $\{e_i:1\leq i \leq k\}$ for the usual generators of $\NN^k$, and $m\vee n$ for the pointwise maximum of $m,n\in \NN^k$. 

In this paper all $k$-graphs are finite in the sense that $\Lambda^n$ is finite for each $n\in \NN^k$. We also assume that our graphs have no sinks and no sources, in the sense that $v\Lambda^n$ and $\Lambda^n v$ are both non-empty for all $v\in \Lambda^0$ and $n\in\NN^k$.    

For $1\leq i\leq k$ we write $A_i$ for the $\Lambda^0\times \Lambda^0$ matrix with entries 
\[A_i(v,w):=|v\Lambda^{e_i}w|.\]
 The factorisation property implies that $A_iA_j=A_jA_i$ for each $1\leq i,j\leq k$. Then for each $n\in \NN^k$,  the matrix $A^n:=\prod_{i=1}^k A_i^{n_i}$ is well defined and has entries given by $A^n(v,w)=|v\Lambda^n w|$. For $C,D\subset \Lambda^0$, we write $A_{C,D,i}$ for the $C\times D$ matrix with entries $A_{C,D,i}(v,w)=A_i(v,w)$ for $v\in C$, $w\in D$, and we write $A_{C,i}:=A_{C,C,i}$. 

The skeleton is the coloured directed graph $(\Lambda^0, \Lambda^1:=\bigcup_{i=1}^k \Lambda^{e_i},r,s)$ in which the edges of degree $e_i$ have been assigned one of $k$ different colours. The relationship between $k$-graphs and skeletons is discussed in \cite[Chapter~10]{R}, \cite[Section~2]{RSY1}, and \cite{HRSW}. When $k=2$,  we think of the elements of $\Lambda^{e_1}$ as blue/solid edges and elements of $\Lambda^{e_2}$ as red/dashed edges. In any $k$-graph, the factorisation property gives bijections $\theta_{i,j}$ of the sets $\{ef:e\in \Lambda^{e_i}, f\in \Lambda^{e_j}\}$ onto $\{gh:g\in \Lambda^{e_j}, h\in \Lambda^{e_i}\}$. However, there are coloured graphs for which there exist such bijections that are not the skeletons of any $k$-graph (see  \cite[Example 5.15(ii)]{KPS}). Fortunately, when $k=2$ the situation is simple: a $2$-coloured graph is the skeleton of a $2$-graph if and only if $A_1 A_2=A_2 A_1$ \cite[Section~6]{KP}. (Though then it is usually the skeleton of many $2$-graphs.)

\subsection{The $C^*$-algebras of higher-rank graphs}
We consider a finite $k$-graph $\Lambda$ with no sinks and no sources. For $\lambda,\mu \in \Lambda$, we define 
\[
\Lambda^{\min}(\lambda,\mu):=\{(\alpha,\beta)\in \Lambda\times \Lambda: \lambda\alpha=\mu\beta\in \Lambda^{d(\lambda)\vee d(\mu)}\}.
\] A collection of partial isometries $\{T_\lambda:\lambda\in\Lambda\}$ in a $C^*$-algebra $B$ is a Toeplitz--Cuntz--Krieger $\Lambda$-family if 
\begin{enumerate}
\item[(T1)] $\{Q_v:=T_v:v\in \Lambda^0\}$ is a collection of mutually orthogonal projections;
\item[(T2)] $T_\lambda T_\mu=T_{\lambda\mu}$ for each $\lambda,\mu\in \Lambda$ with $s(\lambda)=r(\mu)$;
\item[(T3)] $T_\lambda^* T_\mu=\sum_{(\alpha,\beta)\in \Lambda^{\min}(\lambda,\mu)}T_\alpha T_\beta^*$ for all $\lambda,\mu \in \Lambda$. 
\end{enumerate}  
Relation (T3) implies that $T_\lambda^*T_\lambda=Q_{s(\lambda)}$ for  $\lambda\in \Lambda$, and that for each $n\in \NN^k$, the projections $\{T_\lambda T_\lambda^*:\lambda\in \Lambda^n\}$ are mutually orthogonal. Combining this with relation (T2) shows that $Q_v\geq \sum_{\lambda\in v\Lambda^n}T_\lambda T_\lambda^*$ for each $n\in \NN^k$. The relations also imply that $C^*(\{T_\lambda:\lambda\in\Lambda\})=\clsp \{T_\lambda T_\mu^*:\lambda,\mu \in \Lambda\}$. 

The Toeplitz $C^*$-algebra $\TC^*(\Lambda)$ is generated by a universal Toeplitz--Cuntz--Krieger $\Lambda$-family $\{t_\lambda:\lambda\in\Lambda\}$ \cite[\S7]{RS}. The Cuntz--Krieger (or graph) algebra  $C^*(\Lambda)$ is the quotient of $\TC^*(\Lambda)$ in which $q_v=\sum_{\lambda\in v\Lambda^n}t_\lambda t_\lambda^*$ for all $v\in \Lambda^0$ and $n\in \NN^k$.

For $z=(z_1,\dots,z_k)\in \TT^k$ and $n\in \NN^k$, we write $z^n:=\prod_{i=1}^kz_i^{n_i}$.  There is a strongly continuous gauge action $\gamma:\TT^k\rightarrow \Aut(\TC^*(\Lambda))$ satisfying $\gamma_z (t_\lambda)=\prod_{i=1}^k z_i^{d(\lambda)_i}t_\lambda$ for  $z\in \TT^k$ and $\lambda\in \Lambda$. Since $\gamma_z$ fixes each $q_v-\sum_{\lambda\in v\Lambda^n}t_\lambda t_\lambda^*$, this action induces an action on $C^*(\Lambda)$, which we also denote by $\gamma$.

\subsection{Hereditary sets and strongly connected components}
A subset $H\subset \Lambda^0$ is hereditary if $v\in H \text{ and } v\Lambda w\neq \emptyset \Longrightarrow w\in H$, and forwards hereditary if $w\in H \text{ and } v\Lambda w\neq \emptyset \Longrightarrow v\in H$. If $H$ is hereditary then $\Lambda\backslash H:=\{\lambda\in \Lambda:s(\lambda)\not\in H\}$ is a $k$-graph (the key point is that $\mu\nu\in \Lambda\backslash H$ implies $\mu,\nu\in \Lambda\backslash H$). If $I_H$ is the ideal of $\TC^*(\Lambda)$ generated by $\{q_v:v\in H\}$, then $\TC^*(\Lambda\backslash H)$, $\TC^*(\Lambda)/I_H$, and $C^*(\{t_\lambda:s(\lambda)\not\in H\})\subset \TC^*(\Lambda)$ are all canonically isomorphic (see \cite[Proposition~2.2]{aHKR2}). Thus the Toeplitz algebra of the subgraph $\Lambda\backslash H$ can be realised as both a quotient and a subalgebra of $\TC^*(\Lambda)$. 

There is an equivalence relation $\sim$ on $\Lambda^0$ such that 
\[
v\sim w \Longleftrightarrow v\Lambda w\not= \emptyset\text{ and }w\Lambda v \neq \emptyset,
\]
 and  the strongly connected components of $\Lambda$ are the equivalence classes for this relation. If $v\Lambda v=\{v\}$ for a vertex $v$ then $\{v\}$ is a strongly connected component, and we call such classes trivial components. We write $\mathcal{C}$ for the set of nontrivial strongly connected components. 
For $C\in\mathcal{C}$ we set $\Lambda_C:=C\Lambda C$. It is routine to check that if $\mu\nu\in \Lambda_C$, then $\mu,\nu\in \Lambda_C$, and so $\Lambda_C$ is a $k$-graph. 

An $n\times n$ matrix $A$ is irreducible if for every $1\leq i,j\leq n$ there exists $m\in \NN$ such that $A^m(i,j)\neq 0$. If $A$ is the vertex matrix of a directed graph $E$, then $A$ is irreducible if and only if  $E$ is strongly connected. Following \cite{aHKR2}, we say that the subgraph $\Lambda_C$ is coordinatewise irreducible if all the matrices $A_{C,i}$ are irreducible. Then \cite[Lemma~2.1]{aHLRS2} says that the matrices $\{A_{C,i}:1\leq i\leq k\}$ have a common unimodular Perron--Frobenius eigenvector. 

If the graph has no   sources or sinks, and each $\Lambda_C$ is coordinatewise irreducible, \cite[Proposition~3.1]{aHKR2} says that we can order $\Lambda^0$ so that each vertex matrix $A_i$ is block upper triangular,  the diagonal blocks are $\{A_{C,i}:C\in \mathcal{C}\}$, and the other blocks are strictly upper triangular matrices. Hence for every $1\leq i\leq k$, the spectral radius  of $A_i$ is $\rho(A_i)=\max\{\rho(A_{C,i}):C\in \mathcal{C}\}$. 

\subsection{$\KMS$ states}

Suppose that $\alpha$ is an action of $\RR$ on a $C^*$-algebra $A$. An element $a\in A$ is analytic for the action $\alpha$ if the map $t\mapsto \alpha_t(a)$ extends to an analytic function on $\CC$. A state $\phi$ of $A$ is a $\KMS$ state with inverse temperature $\beta\in (0,\infty)$ (or a $\KMS_{\beta}$ state) if $\phi(ab)=\phi(b\alpha_{i\beta}(a))$ for all analytic elements $a,b\in A$. By \cite[Proposition~8.12.3]{Ped} it suffices to verify that the state $\phi$ satisfies the KMS condition on a set of analytic elements which span a dense invariant subset of $A$. 

\subsection{The dynamics on $\TC^*(\Lambda)$}
Let $\Lambda$ be a $k$-graph.  Choosing $r\in (0,\infty)^k$ gives a dynamics $\alpha^r: \RR\rightarrow \Aut(\TC^*(\Lambda))$ by 
\[\alpha^r_t=\gamma_{e^{itr}}:=\gamma_{(e^{itr_1},\ldots, e^{itr_k})}.\]
 For $\lambda,\mu\in \Lambda$, the map $t\mapsto\alpha_t^r(t_\lambda t_\mu^*)=e^{itr\cdot (d(\lambda)-d(\mu))}t_\lambda t_\mu^*$ is the restriction of the analytic function $z\mapsto e^{izr\cdot (d(\lambda)-d(\mu))}t_\lambda t_\mu^*$, and hence $t_\lambda t_\mu^*$ is an analytic element. Since $\lspan\{t_\lambda t_\mu^*:\lambda,\mu \in \Lambda\}$ is dense in $\TC^*(\Lambda)$,  it suffices to check the $\KMS_\beta$ condition on pairs of elements of the form $t_\lambda t_\mu^*$. The action $\alpha^r$ induces an action, also denoted $\alpha^r$, on the quotient $C^*(\Lambda)$. We say that a $\KMS_\beta$ state $\phi$ of $\TC^*(\Lambda)$ factors through $C^*(\Lambda)$ if there exists a $\KMS_\beta$ state $\psi$ of $C^*(\Lambda)$ such that $\phi=\psi\circ q$, where $q$ is the quotient map. 

In Theorem~\ref{dominant2cpt} and beyond, we  assume that the coordinates of the vector $r$ are rationally independent. This implies  that $t\mapsto e^{itr}$ is an embedding of $\RR$ in $\TT^k$. (To see this, suppose that  $e^{itr}=1$. Then there are integers $n_i$ such that $tr_i=2\pi n_i$. But then $r_in_i^{-1}=r_jn_j^{-1}$ and we have $n_jr_i-n_ir_j=0$, which contradicts that the coordinates of $r$ are rationally independent.) It then follows  that the dynamics $\alpha^r$ is not periodic.

 \subsection{Scaling the dynamics}
 We usually scale the dynamics (that is, multiply $r$ by a scalar, which changes the inverse temperature but does not otherwise affect the dynamics \cite[\S2.1]{aHKR}) to ensure that $\beta_c:=\max\{r_i^{-1}\ln\rho(A_i):1\leq i\leq k\}$ is $1$. 
 
 \subsection{The preferred dynamics}
Suppose that $\Lambda$ is a coordinatewise-irreducible finite $k$-graph. Then by \cite[Corollary~4.4 and Theorem~7.2]{aHLRS2}, there is a KMS$_1$ state of $(C^*(\Lambda),\alpha^r)$ if and only if $r_i=\ln\rho(A_i)$ for $1\leq i\leq k$. 
This means $\alpha^r$ is the preferred dynamics, which is characterised by 
\[
\alpha^r(t_\lambda)=\rho(A)^{d(\lambda)}t_\lambda:=\prod_{i=1}^k\rho(A_i)^{d(\lambda)_i}t_\lambda.
\] 
For the preferred dynamics, if the coordinates of $r$ are rationally independent, then $(\TC^*(\Lambda),\alpha^r)$ has a unique KMS$_1$ state, and it factors through a state of $(C^*(\Lambda),\alpha^r)$ \cite[Theorem~7.2]{aHLRS2}.
So if we were interested in KMS$_\beta$ states on $(C^*(\Lambda),\alpha^r)$, then we would only be interested\footnote{Curiously, in her study of $k$-graphs with a single vertex, Yang also arrived at the preferred dynamics, but for different reasons (see \cite[Proposition~5.4]{Yang1}).} in the preferred dynamics and the inverse temperature $\beta=1$.

\section{Removing redundant components}\label{critcpts}

We consider a finite $k$-graph with no sources or sinks.
We suppose that $r\in (0,\infty)^{k}$ satisfies 
\begin{equation}\label{assondyn}
\max\{r_i^{-1}\ln\rho(A_i):1\leq i\leq k\}=1;
\end{equation}
equivalently,
\begin{equation*}
r_i\geq \ln\rho(A_i)\text{ for all $i$, and }K:=\{i:r_i^{-1}\ln\rho(A_i)=1\}\not=\emptyset.
\end{equation*}
 Here it is important that $K$ may be a proper subset of $\{1,\dots,k\}$, in which case the dynamics  $\alpha^r$ on $\TC^*(\Lambda)$ is not the preferred dynamics studied in \cite{aHKR2}. 

In this section we salvage what we can from the arguments used to prove \cite[Theorem~5.1]{aHKR2}. We focus on the  critical components: a nontrivial strongly connected component $C\in\Cc$ such that $r_j=\ln\rho(A_{C,j})$ for some $j\in K$ and $A_{C,j}$ is irreducible.  We then say that $C$ is $j$-critical. The next result says that KMS$_1$ states do not  see many vertices which feed into critical components.

\begin{prop}\label{estnonpref}
Suppose that $\Lambda$ is a finite $k$-graph with no sources or sinks, and that $r\in (0,\infty)^k$ satisfies \eqref{assondyn}.
 Suppose that $C\in \Cc$ is a $j$-critical component of $\Lambda$, and set \[\Sigma_jC:=\{w\in \Lambda^0:C\Lambda^{\NN e_j}w\not=\emptyset\}.\] Then for every KMS$_1$ state $\psi$ on $(\TC^*(\Lambda),\alpha^r)$, we have $\psi(q_w)=0$ for all $w\in H_j:=(\Sigma_jC)\backslash C$.
\end{prop}

\begin{proof}
Since $j\in K$ and  $r_j=\ln\rho(A_{C,j})$, we have $\rho(A_{C,j})=\rho(A_j)$. Thus  \cite[Proposition~4.1(a)]{aHLRS2} implies that the vector  $m^\psi:=\big(\psi(q_v)\big)$ satisfies 
\begin{equation}\label{subinvforj}
A_jm^\psi\leq e^{r_j}m^\psi=\rho(A_j)m^\psi.
\end{equation}
With respect to the decomposition $\Lambda^0=(\Lambda^0\backslash\Sigma_jC)\cup C\cup H_j$, the vertex matrix $A_j$ has block form 
\begin{equation*}
A_j=
\begin{pmatrix}A_{\Lambda^0\backslash\Sigma_jC,j}&\star&\star\\0&A_{C,j}&A_{C,H_j,j}\\0&0&A_{H_j,j}
\end{pmatrix}.  
\end{equation*}
We fix $w\in H_j$. Then there exits $n\geq 1$ and $v\in C$ such that $v\Lambda^{n e_j}w\neq \emptyset$. 

We first suppose that $m^\psi|_C=0$.  Then we have 
\begin{align*}0&\leq A_j^n(v,w)m^\psi_w= (A^n_j)_{C,H_j}(v,w)m_w^\psi\\
&\leq \sum_{u\in H_j} (A_j^n)_{C, H_j}(v,u)m^\psi_u
= \big((A^n_j)_{C,H_j}m^\psi|_{H_j}\big)_v\\
&\leq \big(A^n_jm^\psi\big)_v\leq \rho(A_j)^nm^\psi_v\quad\text{by \eqref{subinvforj}.}
\end{align*}
Since $A_j^n(v,w)>0$, these estimates force $m^\psi_w=0$.

So we suppose that $m^\psi|_C\not=0$. Now we recall that $\rho(A_{C,j})=\rho(A_j)$, and read off from the central block of \eqref{subinvforj} that
\[
A_{C,j}m^\psi|_C\leq (A_jm^\psi)|_C\leq\rho(A_j)m^\psi|_C=\rho(A_{C,j})m^\psi|_C.
\]
Since $A_{C,j}$ is irreducible and $m^\psi|_C\not=0$, the subinvariance theorem \cite[Theorem~1.6]{Sen} implies that
\[
A_{C,j}m^\psi|_C=\rho(A_{C,j})m^\psi|_C. 
\]
With $n$ and $v$  as above we have $A_j^n(v,w)>0$. Then \eqref{subinvforj} gives
\begin{align*}
\rho(A_{C,j})^nm^\psi_v&=\rho(A_j)^nm^\psi_v\geq \big(A^n_jm^\psi\big)_v\\
&=\sum_{u\in C}A_{C,j}^n(v,u)m^\psi_u+\sum_{u\in H_j}A_j^n(v,u)m^\psi_u\\
&\geq\sum_{u\in C}A_{C,j}^n(v,u)m^\psi_u+A_j^n(v,w)m^\psi_w\\
&=\rho(A_{C,j})^nm^\psi_v+A_j^n(v,w)m^\psi_w.
\end{align*}
Thus $m^\psi_w=0$, as required.
\end{proof}

\begin{cor}\label{Ccrit}
Suppose that we have $\Lambda$, $r$, $C$ and $j$ as in Proposition~\ref{estnonpref}, and that the set $H_j$ in that proposition is hereditary in $\Lambda$. Then every KMS$_1$ state of $(\TC^*(\Lambda),\alpha^r)$ factors through a state of $(\TC^*(\Lambda\backslash H_j),\alpha^r)$.
\end{cor}

\begin{proof}
Suppose that $\psi$ is a KMS$_1$ state of $(\TC^*(\Lambda),\alpha^r)$. Proposition~\ref{estnonpref} implies that $\psi$ vanishes on the set $P:=\{q_w:w\in H_j\}$. As at the end of the proof of \cite[Theorem~5.1]{aHKR2}, it follows from \cite[Lemma~2.2]{aHLRS1} that $\psi$ vanishes on the ideal $I_{H_j}$ generated by $P$. Since $H_j$ is hereditary, \cite[Proposition~2.2]{aHKR2} implies that $\TC^*(\Lambda)/I_{H_j}\cong\TC^*(\Lambda\backslash H_j)$, and the result follows. 
\end{proof}

Corollary~\ref{Ccrit} applies in particular if $C\Lambda^{\NN e_j}w\not=\emptyset$ for all vertices $w$, and then says that every KMS$_1$ state of $(\TC^*(\Lambda),\alpha^r)$ factors though a state of $(\TC^*(\Lambda_C),\alpha^r)$. However, it is possible that $H_j=\emptyset$. For example, suppose that $\Lambda$ is a dumbbell graph as in Example~\ref{ex2dumbbell} with $p_1=0$ (and then necessarily with $m_1=n_1$). Then Corollary~\ref{Ccrit} gives no information.

When the set $H_j$ of Proposition~\ref{estnonpref} is not hereditary,  
$\Lambda\backslash H_j$ may not be a $k$-graph. Our next goal is to prove that, provided  $\Lambda$ is sufficiently connected,  there exists a 
hereditary subset $H$ such that  every KMS$_1$ state of $\TC^*(\Lambda)$ 
factors through a KMS$_1$ state of $\TC^*(\Lambda\backslash H)$. We start with a lemma.

\begin{lem}\label{allpossame}
Suppose that $\Lambda$ is a finite $k$-graph with no sources or sinks. Suppose that $\bigcup\{C:C\in \Cc\}=\Lambda^0$, that all the graphs $\{\Lambda_C:C\in \Cc\}$ are coordinatewise irreducible and that for 
components $C,D\in \Cc$ we have
\begin{equation}\label{allconnected}
C\Lambda^{\NN e_j}D\not=\emptyset \quad\text{for some $j$\ 
}\Longrightarrow \ C\Lambda^{\NN e_i}D\not=\emptyset\quad\text{for $1\leq 
i\leq k$.}
\end{equation}
Suppose that $v\leq w$, in the sense that $v\Lambda w\not=\emptyset$. Then $v\Lambda^{\NN e_i}w\not=\emptyset$ for $1\leq i\leq k$.
\end{lem}

\begin{proof}
Suppose that $\lambda\in v\Lambda w$ and  $1\leq i\leq k$. Since $\Lambda^0=\bigcup\{C:C\in \Cc\}$, we can factor $\lambda=\mu_1\lambda_1\dots\mu_m\lambda_m\mu_{m+1}$, where each $\mu_i$ 
has range and source in the same component and each $\lambda_i$ is an 
edge with range and source in different components. For $1\leq n\leq m+1$, let $C_n$ be the component such that $\mu_n\in \Lambda_{C_n}$. Since $\Lambda_{C_n}$ is coordinatewise irreducible, there are paths $\mu_n'\in \Lambda_{C_n}^{\NN e_i}$ such that $r(\mu_n')=r(\mu_n)$ and $s(\mu_n')=s(\mu_n)$. Let $D_n$ be the components such that $\lambda_n\in C_n\Lambda^{e_j}D_n$ for some~$j$. Then \eqref{allconnected} implies that there exists $\lambda_n'\in C_n\Lambda^{\NN e_i}D_n$; since both $\Lambda_{C_n}$ and $\Lambda_{D_n}$ are coordinatewise irreducible, we can suppose also that $\lambda_n'\in r(\lambda)\Lambda^{\NN e_i}s(\lambda_n)$. Then $\mu_1'\lambda_1'\dots\mu_m'\lambda_m'\mu_{m+1}'$ is a path in $v\Lambda^{\NN e_i}w$.
\end{proof}

We consider the partial order on components such that $C\leq D\Longleftrightarrow C\Lambda D\not=\emptyset$. This restricts to a partial order on the set $\Ccrit$ of critical components. Since $\Ccrit\subset \Cc$ is finite, there are elements of $\Ccrit$ which are minimal in this partial order. We denote the set of these minimal elements by $\Cmcrit$.

\begin{prop}\label{reducetocrither}
Suppose that $\Lambda$ is a finite $k$-graph with no sources or sinks, and that $r\in (0,\infty)^k$ satisfies \eqref{assondyn}. Suppose that $\bigcup\{C:C\in \Cc\}=\Lambda^0$, that all the graphs $\{\Lambda_C:C\in \Cc\}$ are coordinatewise irreducible, and that \eqref{allconnected} holds.  Let $H'$ be the hereditary closure of $G:=\bigcup\{C:C\in \Cmcrit\}$, and set $H:=H'\backslash G$. Then $H$ is hereditary, $\Cmcrit$ is the set of all critical components in the graph $\Lambda\backslash H$, and each $C\in \Cmcrit$ is hereditary in $\Lambda\backslash H$. Every KMS$_1$ state of $(\TC^*(\Lambda),\alpha^r)$ factors through a KMS$_1$ state of $(\TC^*(\Lambda\backslash H),\alpha^r)$.
\end{prop}

\begin{proof}
To see that $H$ is hereditary, suppose that $v\in H$ and $v\leq w$. Since $v$ belongs to the hereditary closure of $G$, there exists $C\in \Cmcrit$ such that $C\leq v$. But then $C\leq w$ also, and $w\in H'$.  Suppose, looking for a contradiction, that $w\in G$. Then there exists $D\in \Cmcrit$ such that $w\in D$. Then $C\Lambda D\neq \emptyset$, and minimality forces $C=D$.
Then $w\in C$  forces $v\in C$, which is not possible because $v\in H=H'\backslash G$. So $w$ is not in $G$, and $w\in H$. Thus $H$ is hereditary.

Since we have removed all the critical components that are not minimal, and have not added any new ones, the critical components of $\Lambda\backslash H$ are those in $\Cmcrit$. To see that $C\in \Cmcrit$ is hereditary in $\Lambda\backslash H$, take $v\in (\Lambda\backslash H)^0=\Lambda^0\backslash H$ such that $C\Lambda v\not=\emptyset$. Then $v$ belongs to the hereditary closure $H'$, and since $H'=G\cup H$, we have $v\in G$. So there exists $D\in \Cmcrit$ such that $v\in D$. Then $C\leq D$, and minimality of $D$ forces $D=C$. Thus $v\in C$, as required. 

Now we suppose that $\phi$ is a KMS$_1$ state of $(\TC^*(\Lambda),\alpha^r)$. Let $w\in H$. Then there exists $C\in\Cmcrit$ such that $C\Lambda w\not=\emptyset$. Then $C$ is $j$-critical for some $j$. By Lemma~\ref{allpossame}, $C\Lambda^{\NN e_j} w\neq\emptyset$. Then  $w$ belongs to the set $\Sigma_jC$ of Proposition~\ref{estnonpref}, and that proposition implies that $\phi(q_w)=0$. Thus $\phi(q_w)=0$ for all $w\in H$. We deduce from \cite[Lemma~6.2]{AaHR} (for example\footnote{We could also use \cite[Lemma~2.2]{aHLRS1} for this, but \cite[Lemma~6.2]{AaHR} is easier to use and much more general.}) that $\phi$ vanishes on the ideal generated by $\{q_w:w\in H\}$. This ideal is $I_H$, and \cite[Proposition~2.2]{aHKR2} shows that the canonical map $q_H:\TC^*(\Lambda)\to \TC^*(\Lambda\backslash H)$ induces an isomorphism of $\TC^*(\Lambda)/I_H$ onto $\TC^*(\Lambda\backslash H)$. Thus $\phi$ factors through a state of $\TC^*(\Lambda\backslash H)$ which is necessarily a KMS$_1$ state of $(\TC^*(\Lambda\backslash H),\alpha^r)$.
\end{proof}

\begin{rmk}
The hypotheses of Proposition~\ref{reducetocrither} imply that the graph $\Lambda\backslash H$ also has no sources or sinks. In \cite[\S8]{aHKR2}, we showed how removing a hereditary component can create sources and hence extra KMS states; Examples~8.4 and 8.5 in \cite{aHKR2} show that quite a variety of things can happen if there are complicated bridges between components. Here, where we are focusing on graphs which have no long bridges between components, Proposition~\ref{reducetocrither} says that it suffices to study the KMS$_1$ states for graphs in which every critical component is hereditary.
\end{rmk}

\section{Extending the Perron--Frobenius eigenvector}\label{extendPF}

The following is a strengthening of \cite[Proposition~6.1]{aHKR2}, which is a linear-algebraic result about the vertex matrices of a finite $k$-graph. We recover \cite[Proposition~6.1]{aHKR2} when the set $L_D$ at \eqref{Dhered-new} is $\{1,\dots ,k\}$. In the applications, the eigenvector $x$ will be the common unimodular Perron--Frobenius eigenvector of the commuting irreducible matrices $\{A_{D,i}\}$. It is important for the application in Theorem~\ref{orderrho} that the set $L_D$ in \eqref{Dhered-new} can be a proper subset of $\{1,\dots,k\}$.

\begin{prop}\label{evectors-new}
Suppose that $\Lambda$ is a finite $k$-graph without sinks or sources. Suppose that $D\in \Cc$ is hereditary, and set $H:=\{v\in \Lambda^0:v\Lambda D=\emptyset\}$. Suppose that
\begin{equation}\label{Dhered-new}
L_D:=\big\{i\in \{1,\dots,k\}:\rho(A_{C,i})<\rho(A_{D,i})\quad\text{for $C\in \Cc\backslash \{D\}$ with $C\Lambda D\not=\emptyset$}\big\}
\end{equation}
is nonempty, and that $x$ is a nonnegative eigenvector of $A_{D,i}$ with eigenvalue $\rho(A_{D,i})$ for $1\leq i\leq k$.  Write $F:=\Lambda^0\backslash(D\cup H)$. Then with respect to the decomposition $\Lambda^0=F\sqcup D\sqcup H$, the  vertex matrices have block form
\begin{equation*}
A_i=\begin{pmatrix} E_i & B_i&\star\\ 0 & A_{D,i}&0\\0&0&A_{H,i}\end{pmatrix}\quad\text{for $1\leq i\leq k$.} 
\end{equation*}
\begin{enumerate}
\item\label{equaly} For $1\leq i,j\leq k$ we have
\begin{equation}\label{easyeq}
(\rho(A_{D,i})1_F-E_i)B_j x=(\rho(A_{D,j})1_F-E_j)B_ix.
\end{equation}
\item\label{defy2} For $i,j\in L_D$, we have 
\begin{equation}\label{defy-new}
(\rho(A_{D,i})1_F-E_i)^{-1}B_i x=(\rho(A_{D,j})1_F-E_j)^{-1}B_jx;
\end{equation}
write $y$ for the common vector \eqref{defy-new}. Then $y$ is nonnegative, and for every $i\in \{1,\dots,k\}$, $(y,x,0)$ is an eigenvector of $A_i$ with eigenvalue $\rho(A_{D,i})$.
\end{enumerate}
\end{prop}

\begin{proof}
Since $D$ is hereditary, we have $D\Lambda(\Lambda^0\backslash D)=\emptyset$, and the block decomposition has the required form.

Suppose that $i,j\in \{1,\dots ,k\}$. Since $A_iA_j=A_jA_i$, the block form of the product gives 
$E_iE_j=E_jE_i$ and 
\begin{equation}\label{eq:case2_matrix}
E_iB_j+B_iA_{D,j}=E_jB_i+B_jA_{D,i}. 
\end{equation}
Since $x$ is an eigenvector for both $A_{D,i}$ and $A_{D,j}$, \eqref{eq:case2_matrix} gives
\begin{align*}
(\rho(A_{D,i})1_F-E_i)B_j x &=\rho(A_{D,i})B_jx-E_iB_jx\\
&=\rho(A_{D,i})B_jx-(E_jB_ix+B_jA_{D,i}x-B_iA_{D,j}x)\notag\\
&=\rho(A_{D,i})B_jx-(E_jB_ix+B_j\rho(A_{D,i})x-B_iA_{D,j}x)\notag\\
&=B_iA_{D,j}x-E_jB_ix\notag\\
&=(\rho(A_{D,j})1_F-E_j)B_ix.\notag
\end{align*}
Thus we have \eqref{easyeq}, and we have proved \eqref{equaly}.

Next we take $i,j\in L_D$. Then 
\[
\rho(A_{D,i})>\max\big\{\rho(A_{C,i}):C\in \Cc\backslash\{D\}\text{ and }C\Lambda D\not=\emptyset\big\}=\rho(E_i), \]
and similarly for $j$, so the matrices $\rho(A_{D,i})1_F-E_i$ and $\rho(A_{D,j})1_F-E_j$ are invertible. They also commute, and hence so do their inverses. So we can multiply \eqref{easyeq} by 
\[
(\rho(A_{D,i})1_F-E_i)^{-1}(\rho(A_{D,j})1_F-E_j)^{-1}
\]
to get \eqref{defy-new}. Thus we can define $y$ as claimed. Since $x\geq 0$ and $B_i$ has nonnegative entries, the expansion
\begin{align*}
(\rho(A_{D,i})1_F-E_i)^{-1}&=\rho(A_{D,i})^{-1}(1_F-\rho(A_{D,i})^{-1}E_i)^{-1}\\
&=\rho(A_{D,i})^{-1}\sum_{n=0}^\infty\rho(A_{D,i})^{-n}E^n_i
\end{align*}
shows that $y$ is nonnegative. 
To prove that $z:=(y,x,0)$ is an eigenvector of every vertex matrix~$A_j$ with eigenvalue $\rho(A_{D,j})$, we fix $j$. Then
\[
A_jz=\begin{pmatrix} E_i & B_i&\star\\ 0 & A_{D,i}&0\\0&0&A_{H,i}\end{pmatrix}\begin{pmatrix}
y\\x\\0\end{pmatrix}=\begin{pmatrix}
E_jy+B_jx\\A_{D,j}x\\0\end{pmatrix}=\begin{pmatrix}
E_jy+B_jx\\ \rho(A_{D,j})x\\0\end{pmatrix}.
\]
We now take $i\in L_D$, and work on the top block in $A_jz$. Since $E_iE_j=E_jE_i$ it follows that $E_j$ and $(\rho(A_{D,i})1_F-E_i)^{-1}$ commute. Thus 
\begin{align*}
E_jy+B_jx&=(\rho(A_{D,i})1_F-E_i)^{-1}E_jB_ix+B_jx\\
&=(\rho(A_{D,i})1_F-E_i)^{-1}\big(\rho(A_{D,j})1_F-(\rho(A_{D,j})1_F-E_j)\big)B_ix+B_jx.
\end{align*}
At this point we use \eqref{easyeq}, finding
\begin{align*}
E_jy+B_jx&=(\rho(A_{D,i})1_F-E_i)^{-1}\big(\rho(A_{D,j})B_ix-(\rho(A_{D,i})1_F-E_i)B_jx\big)+B_jx\\
&=\rho(A_{D,j})(\rho(A_{D,i})1_F-E_i)^{-1}B_ix-B_jx+B_jx\\
&=\rho(A_{D,j})y.
\end{align*}
Thus $z$ is an eigenvector of every $A_j$ with eigenvalue $\rho(A_{D,j})$, and this completes the proof of \eqref{defy2}.
\end{proof}

The next application of Proposition~\ref{evectors-new} has no analogue in \cite{aHKR2}. We find it curious that the proof of Theorem~\ref{orderrho} involves a non-trivial application of the subinvariance theorem from Perron--Frobenius theory, and wonder whether a more direct linear-algebraic proof is possible. We give such a proof for graphs with three components in Appendix~\ref{app3vertex}. (If so, we could then deduce the stronger-looking Proposition~\ref{evectors-new} from \cite[Proposition~6.1]{aHKR2}.) 

In the statement of Theorem~\ref{orderrho}, we have removed the hereditary set $H$ from the set-up of Proposition~\ref{evectors-new}; if there is such a set, we can apply Theorem~\ref{orderrho} to $\Lambda\backslash  H$, but then we only get information about $\rho(A_{C,i})$ for components $C\in \Cc$ with $C\subset \Lambda^0\backslash H$.

\begin{thm}\label{orderrho}
Suppose that $\Lambda$ is a finite $k$-graph without sinks or sources, and that the subgraphs $\{\Lambda_C:C\in \Cc\}$ are all coordinatewise irreducible. Suppose that $D\in \Cc$ is hereditary, and that $C\Lambda^{e_i}D\not=\emptyset$ for all $C\in \Cc$ and $1\leq i\leq k$. If there exists $j\in\{1,\dots,k\}$ such that $\rho(A_{C,j})<\rho(A_{D,j})$ for all $C\in \Cc\backslash \{D\}$, then
\begin{equation*}
\rho(A_{C,i})<\rho(A_{D,i})\quad\text{for all $C\in \Cc\backslash \{D\}$ and $1\leq i\leq k$.}
\end{equation*}
\end{thm}

\begin{proof}
Since $\Lambda_D$ is coordinatewise irreducible, its vertex matrices $\{A_{D,i}:1\leq i\leq k\}$ are a commuting family of irreducible matrices, and hence by \cite[Lemma~2.1]{aHLRS2} have a common unimodular Perron--Frobenius eigenvector $x$. By assumption,  the set $L_D$ in Proposition~\ref{evectors-new} is nonempty, and hence that proposition gives a vector $y\in [0,\infty)^{\Lambda^0\backslash D}$ such that $z:=(y,x)$ is an eigenvector of each $A_{i}$ with eigenvalue $\rho(A_{D,i})$.

Now we take $C\in \Cc\backslash\{D\}$, and consider the block $y|_C$. Fix $i\in\{1,\dots,k\}$. The $C\times C$ block of $E_i$ is $A_{C,i}$, and the $C\times D$ block of $B_i$ is $A_{C,D,i}$. Thus
\begin{equation}\label{z|Csubinv}
A_{C,i}y|_C\leq A_{C,i}y|_C+A_{C,D,i}x\leq (A_iz)|_C=\rho(A_{D,i})z|_C=\rho(A_{D,i})y|_C.
\end{equation}
Since $C\Lambda^{e_i}D\not=\emptyset$, and since $x$ has strictly positive entries, the vector $A_{C,D,i}x$ is nonzero. So $y|_C\not=0$ and the first inequality in \eqref{z|Csubinv} is strict. So \eqref{z|Csubinv} implies that $y|_C$ is a subinvariant vector for the irreducible matrix $A_{C,i}$, and that it is not an eigenvector. So the subinvariance theorem \cite[Theorem~1.6]{Sen} implies that $\rho(A_{C,i})<\rho(A_{D,i})$.
\end{proof}

The next corollary strengthens Theorem~\ref{orderrho} by allowing longer singly-coloured bridges between components.

\begin{cor}
\label{orderrho2}
Suppose that $\Lambda$ is a finite $k$-graph without sinks or sources, and that the subgraphs $\{\Lambda_C:C\in \Cc\}$ are all coordinatewise irreducible. Suppose  that $D\in \Cc$ is hereditary and satisfies $C\Lambda^{\NN e_i}D\not=\emptyset$ for all $C\in \Cc$ and $1\leq i\leq k$. If there exists $j\in\{1,\dots,k\}$ such that $\rho(A_{C,j})<\rho(A_{D,j})$ for all $C\in \Cc\backslash \{D\}$, then
\begin{equation*}
\rho(A_{C,i})<\rho(A_{D,i})\quad\text{for all $C\in \Cc\backslash \{D\}$ and $1\leq i\leq k$.}
\end{equation*}
\end{cor}

\begin{proof}
For each $C\in \Cc$ and $i$ there exists $N_{C,i}\in \NN$ such that $C\Lambda^{N_{C,i} e_i}D\not=\emptyset$. Since $\Lambda_C$ and $\Lambda_D$ are coordinatewise irreducible, we then have $C\Lambda^{p e_i}D\not=\emptyset$ for all $p\geq N_{C,i}$. Thus there exist $N_i\in \NN$ such that $C\Lambda^{N_i e_i}D\not=\emptyset$ for all $C$. (If some $A_{C,i}$ is a permutation matrix, then also choose $N_i$ coprime to $|C|$, so that $A_{C,i}^{N_i}$ is also irreducible.) Now we set $N:=(N_1,\dots,N_k)$ and  
\[
N\NN^k:=\big\{Nn:=(N_1n_1,\dots,N_kn_k):n\in \NN^k\big\}. 
\]
Then $\Lambda(N):=\{\lambda\in \Lambda:d(\lambda)\in N\NN^k\}$ is a $k$-graph, 
with the same range and source maps as $\Lambda$  and degree functor $d(N):\Lambda(N)\to \NN^k$ defined by $Nd(N)(\lambda)=d(\lambda)$. The choice of the $N_i$ ensures that $\Lambda(N)$ satisfies the hypotheses of Theorem~\ref{orderrho}. Since the vertex matrices $A(N)_i$ of $\Lambda(N)$ satisfy $A(N)_{C,i}=A_{C,i}^{N_i}$, applying Theorem~\ref{orderrho} to $\Lambda(N)$ gives the result.
\end{proof}

\begin{example}\label{ex2dumbbell}
We consider a $2$-graph $\Lambda$ with skeleton of the form shown in Figure~\ref{2dumbbell} (known as a dumbbell graph).
\begin{figure}[h]
\begin{tikzpicture}[scale=1.5]
 \node[inner sep=0.5pt, circle] (v) at (0,0) {$v$};
    \node[inner sep=0.5pt, circle] (w) at (2,0) {$w$};
    \draw[-latex, blue] (v) edge [out=320, in=220, loop, min distance=30, looseness=2.5] (v);
\draw[-latex, red, dashed] (v) edge [out=40, in=140, loop, min distance=30, looseness=2.5] (v);
\draw[-latex, blue] (w) edge [out=220, in=320, loop, min distance=30, looseness=2.5] (w);
\draw[-latex, red, dashed] (w) edge [out=140, in=40, loop, min distance=30, looseness=2.5] (w);
\draw[-latex, blue] (w) edge [out=195, in=345]   (v);
\draw[-latex, red, dashed] (w) edge [out=165, in=15]   (v);
\node at (2.55, 0.4) {$n_2$};
\node at (2.55,-0.4) {$n_1$};
\node at (1.0,-.35) {$p_1$};
\node at (1.0,.35) {$p_2$}; 
\node at (-0.55,-0.4) {$m_1$};
\node at (-.55, 0.4) {$m_2$};
\end{tikzpicture}
\caption{A dumbbell $2$-graph with $2$ components.}\label{2dumbbell}
\end{figure}
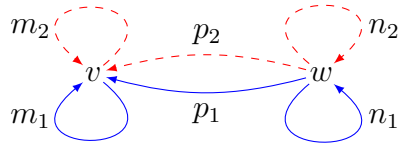
With the vertex set ordered alphabetically, the vertex matrices of $\Lambda$ are
\[
A_1=\begin{pmatrix}m_1&p_1\\0&n_1\end{pmatrix}\quad\text{and }\quad 
A_2=\begin{pmatrix}m_2&p_2\\0&n_2\end{pmatrix}.
\]
The factorisation property implies that $A_1A_2=A_2A_1$, and
\begin{align}
\notag(A_1A_2)(v,w)=(A_2A_1)(v,w)&\Longleftrightarrow m_1p_2+p_1n_2=m_2p_1+p_2n_1\\
\label{condforskel-new}&\Longleftrightarrow (n_2-m_2)p_1=(n_1-m_1)p_2.
\end{align}
Thus the graph in Figure~\ref{2dumbbell} is the skeleton of a $2$-graph if and only if \eqref{condforskel-new} holds.

If $p_1$ and $p_2$ are both nonzero, then \eqref{condforskel-new} implies that $n_1< m_1\Longleftrightarrow n_2< m_2$, as predicted by Theorem~\ref{orderrho}. However, if exactly one $p_i$ is zero, say $p_1=0$, then \eqref{condforskel-new} implies that $m_1=n_1$ and imposes no restriction on $m_2$ and $n_2$. So we cannot remove the hypotheses $C\Lambda^{e_i}D\not =\emptyset$ from Theorem~\ref{orderrho} or $C\Lambda^{\NN e_i}D\not=\emptyset$ from Corollary~\ref{orderrho2} (though we could possibly weaken them if we wanted to allow trivial strongly connected components, as in \cite[\S8]{aHKR2}, for example).
\end{example}

\section{Dominant components}\label{sec:dom}

The other main general result in \cite{aHKR2} is Theorem~6.5, which describes the KMS$_1$ states of $\TC^*(\Lambda)$ when there is a hereditary component where all the spectral radii are attained. Here we seek a version of \cite[Theorem~6.5]{aHKR2} for a non-preferred dynamics.  

From now on we assume that \[\rho(A_{C, i})>1\text{\ for  $1\leq i\leq k$ and $C\in \Cc$}.\]  Since we usually assume that the graphs $\Lambda_C$ are coordinatewise irreducible, the assumption $\rho(A_{C,i})>1$ merely removes the possibility that $\Lambda_{C,i}$ consists of a single cycle, which even for $1$-graphs is known to be an exceptional case.

\begin{prop}\label{extendKMS}
Suppose that $\Lambda$ is a finite $k$-graph with no sources or sinks, and choose a vector $r\in (0,\infty)^{k}$  satisfying \eqref{assondyn}. We suppose that $D\in \Cc$ is hereditary, that $\Lambda_D$  is coordinatewise irreducible, and that there exists $j\in\{1,\dots, k\}$ such that
\begin{equation}\label{convA}\rho(A_{C,j})<\rho(A_{D,j})\quad\text{for $C\in \Cc\backslash\{D\}$ such that $C\Lambda D\not=\emptyset$.}
\end{equation}
We let $x$ be the common Perron--Frobenius eigenvector of the $\{A_{D,i}:1\leq i\leq k\}$, and take $z=(y,x,0)$ as in Proposition~\ref{evectors-new}. Write $b:=\|z\|_1$. Then there is a KMS$_1$ state $\psi$ of $(\TC^*(\Lambda),\alpha^r)$ such that
\begin{equation}\label{defpsi}
\psi(t_\mu t_\nu^*)=\delta_{\mu,\nu}e^{-r\cdot d(\mu)}b^{-1}z_{s(\mu)}.
\end{equation}
This state factors through a state of $C^*(\Lambda)$ if and only if $r_i=\ln\rho(A_{D,i})$ for all $i$.
\end{prop}

\begin{proof}[Outline of proof]
Since the argument is very similar to that of \cite[Proposition~6.3]{aHKR2}, we merely outline the argument.
Since $\{A_{D,i}:1\leq i\leq k\}$ is a commuting family of irreducible nonnegative matrices, Lemma~2.1 of \cite{aHLRS2} implies that they have a common unimodular Perron--Frobenius eigenvector $x\in (0,\infty)^D$.  Let $\Lambda^0=F\sqcup D\sqcup H$ be the decomposition of Proposition~\ref{evectors-new}. The proposition then implies that there is a vector $y\in[0,\infty)^{F}$ such that $z:=(y,x,0)$ is a common eigenvector of each $A_i$ with eigenvalue $\rho(A_{D,i})$. 

Following the proof of \cite[Proposition~6.3]{aHKR2}, we choose a decreasing sequence $\{\beta_p\}\subset (1,\infty)$ such that $\beta_p\to 1$ as $p\to \infty$. Then 
\[
\beta_pr_i\geq \beta_p\ln\rho(A_i)>\ln \rho(A_i)\geq \ln\rho(A_{D,i})\quad\text{ for all $i$,}
\]
and hence
\[
\epsilon^p:=\prod_{i=1}^k\big(1-e^{-\beta_pr_i}A_i\big)b^{-1}z=\prod_{i=1}^k\big(1-e^{-\beta_pr_i}\rho(A_{D,i})\big)b^{-1}z\geq 0.
\]
Then \cite[Theorem~6.1]{aHLRS2} gives KMS$_{\beta_p}$ states $\phi_{\epsilon^p}$ of $(\TC^*(\Lambda),\alpha^r)$, from which a weak* compactness argument gives a KMS$_1$ state $\psi$ of $(\TC^*(\Lambda),\alpha^r)$ such that
\begin{equation*}
\psi(t_\mu t^*_\nu)=\lim_{p\to\infty}\delta_{\mu,\nu}e^{-\beta_p r\cdot d(\mu)}b^{-1}z_{s(\mu)}=\delta_{\mu,\nu}e^{-r\cdot d(\mu)}b^{-1}z_{s(\mu)}.
\end{equation*}

By \cite[Lemma~6.2]{AaHR}, $\psi$ factors through $q:\TC^*(\Lambda)\to C^*(\Lambda)$  if and only if \[
\psi\Big(q_v-\sum_{\lambda\in v\Lambda^n}t_{\lambda} t_{\lambda}\Big)=0\quad\text{for all $v\in \Lambda^0$ and $n\in \NN^k$.}
\]
Thus for the last comment, we take $v\in \Lambda^0$, $n\in \mathbb{N}^k$ and compute:
 \begin{align*}
\psi\bigg(\sum_{\lambda\in v\Lambda^n}t_\lambda t_\lambda^*\bigg)
&=\sum_{\lambda\in v\Lambda^n}e^{-r\cdot n}b^{-1}z_{s(\lambda)}
=\sum_{w\in \Lambda^0} \bigg(\prod_{i=1}^k A_i^{n_i}\bigg)(v,w)e^{-r\cdot n}b^{-1}z_w\\
&=e^{-r\cdot n}b^{-1}\bigg(\bigg(\prod_{i=1}^k A_i^{n_i}\bigg)z\bigg)_v=e^{-r\cdot n}b^{-1}\prod_{i=1}^k \rho(A_{D,i})^{n_i}z_v\notag\\
&=b^{-1}\prod_{i=1}^k (e^{-r_i}\rho(A_{D,i}))^{n_i}z_v
=\prod_{i=1}^k (e^{-r_i}\rho(A_{D,i}))^{n_i}\psi(q_v).\notag
\end{align*}
It follows that $\psi$ factors through $q$ if and only if $\prod_{i=1}^k (e^{-r_i}\rho(A_{D,i}))^{n_i}=1$ for every $n\in \NN^k$. Since each $e^{-r_i}\rho(A_{D,i})\leq 1$, this last condition holds if and only if  $r_i=\ln\rho(A_{D,i})$ for all $i$.
\end{proof}

We can now state our new version of \cite[Theorem~6.5]{aHKR2}. Notice that the hypothesis \eqref{convAstrong} is substantially stronger than the hypothesis \eqref{convA} in Proposition~\ref{extendKMS}; it implies that there are no other critical components. This hypothesis was crucial in the proof of \cite[Theorem~6.5]{aHKR2}, which we follow.

\begin{thm}\label{dominant2cpt}
Suppose that $\Lambda$ is a finite $k$-graph with no sources or sinks, and choose a vector $r\in (0,\infty)^{k}$  satisfying \eqref{assondyn}, and with rationally independent coordinates. 
We suppose that $D\in \Cc$ is hereditary, that $\Lambda_D$ is coordinatewise irreducible, that there exists $j$ such that $r_j=\ln\rho(A_{D,j})$, and such that
\begin{equation}\label{convAstrong}
\rho(A_{C,i})<\rho(A_{D,i})\quad\text{for $1\leq i\leq k$ and  $C\in \Cc\backslash\{D\}$.}
\end{equation}
Write $q_D$ for the quotient map of $\TC^*(\Lambda)$ onto $\TC^*(\Lambda\backslash D)$. Then every KMS$_1$ state of $(\TC^*(\Lambda),\alpha^r)$ is a convex combination of the state $\psi$ of Proposition~\ref{extendKMS} and a state $\phi\circ q_D$ lifted from a KMS$_1$ state $\phi$ of $(\TC^*(\Lambda\backslash D),\alpha^r)$. \end{thm}

\begin{proof}
We begin by observing that we can swap the colours around to ensure that there exists $k'$ such that 
\[
r_i=\ln\rho(A_i)\Longleftrightarrow k'\leq i\leq k.
\]
In particular, we then have $r_k=\ln\rho(A_k)=\ln\rho(A_{D,k})$. 

Now suppose that $\theta$ is a KMS$_1$ state of $(\TC^*(\Lambda),\alpha^r)$. Then \cite[Proposition~4.1]{aHLRS2} implies that the vector $m^\theta:=(\,\theta(q_v)\,)\in [0,1]^{\Lambda^0}$ satisfies $A_im^\theta\leq e^{r_i}m^\theta$ for all $i$. Looking at the block structure of $A_i$ for the decomposition $\Lambda^0=(\Lambda^0\backslash D)\sqcup D$ shows that the vector $m^\theta|_D$ satisfies 
\[
A_{D,i}\big(m^\theta|_D\big)\leq e^{r_i}m^\theta|_D\quad\text{for $1\leq i\leq k$,}
\]
and hence
\[
A_{D,i}\big(m^\theta|_D\big)\leq \rho(A_i)m^\theta|_D=\rho(A_{D,i})m^\theta|_D\quad\text{for $k'\leq i\leq k$.}
\]
Since each $A_{D,i}$ is irreducible, the subinvariance theorem \cite[Theorem~1.6]{Sen} implies that 
\[
A_{D,i}\big(m^\theta|_D\big)=\rho(A_{D,i})m^\theta|_D\quad\text{for $k'\leq i\leq k$.}
\]
Since the matrices $\{A_{D,i}:1\leq i\leq k\}$ have just one common unimodular Perron--Frobenius eigenvector $x$, we deduce that there exists $a\in [0,\infty)$ such that $m^\theta|_D=ab^{-1}x$, for $b:=\|z\|_1$ as in Proposition~\ref{extendKMS}.

We could have $a=0$, in which case it follows from \cite[Lemma~6.2]{AaHR} that $\theta$ factors through $q_D$. We aim to prove that $a\leq 1$, as in the third and fourth paragraphs of the proof of \cite[Theorem~6.5]{aHKR2}, by showing that $\theta(q_v)\geq a\psi(q_v)$ for all $v$. As in \cite{aHKR2}, the interesting case is when $v$ belongs to a component $C$ with $C\Lambda D\not=\emptyset$, and for this we follow the calculation in the fourth paragraph, working in the coordinate graph $\Lambda_k$. Now we could have $a=1$, and then we get $\theta=\psi$, as in \cite{aHKR2}. (This uses the direction of \cite[Proposition~3.1(b)]{aHLRS2} which requires rational independence of the~$r_i$.)

So we are left with the case $0<a<1$, in which case we have to construct a KMS$_1$ state $\phi_\epsilon$ of $\TC^*(\Lambda\backslash D)$ by applying \cite[Theorem~6.1]{aHLRS2} to the graph $\Lambda\backslash D$. By \eqref{convAstrong} we have $r_i\geq \ln\rho(A_i)\geq \ln\rho(A_{D,i})>\ln\rho(A_{C,i})$ for $1\leq i\leq k$ and $C\in\Cc\backslash\{D\}$, and hence $r_i>\ln\rho(A_{\Lambda^0\backslash D,i})$ for all $i$.
Therefore KMS$_1$ states of $\TC^*(\Lambda\backslash D)$ have the form $\phi_\epsilon$ as described in \cite[Theorem~6.1]{aHLRS2}.\footnote{Since $\Lambda\backslash D$ could have sources, this application  depends on the observation at the start of \cite[\S8]{aHKR2} that \cite[Theorem~6.1]{aHLRS2} applies also to graphs with sources.}

Define $\kappa:=(1-a)^{-1}(m^\theta-am^\psi)|_{\Lambda^0\backslash D}$, write each $A_i$ in block form with $E_i:=A_{\Lambda^0\backslash D,i}$, and take 
\[
\epsilon:=\prod_{i=1}^k\big(1-e^{-r_i}E_i\big)\kappa.
\]
The argument in the seventh paragraph of the proof in \cite{aHKR2} shows that $\epsilon\geq 0$. (This is the reason we swapped the colours around at the start: the $k$th matrix $A_k$ plays a special role in that calculation, and it is crucial that $\rho(A_k)=\rho(A_{D,k})$.) We have
\begin{align*}
\|\kappa\|_1&=(1-a)^{-1}\Big(\sum_{v\in\Lambda^0\backslash 
D}\theta(q_v)- a\sum_{v\in\Lambda^0\backslash D}\psi(q_v)\Big)\\
&=(1-a)^{-1}\Big(1-\sum_{v\in D}\theta(q_v)-a\big(1-\sum_{v\in 
D}\psi(q_v)\big)\Big)\\
&=(1-a)^{-1}\Big(1-\sum_{v\in D}ab^{-1}x_v -a+a\sum_{v\in D}b^{-1}x_v \Big)=1.
\end{align*}
Thus $\epsilon$ belongs to the simplex $\Sigma_1$ of 
\cite[Theorem~6.1]{aHLRS2}. We then finish off by following the argument in the last two paragraphs of the proof in~\cite{aHKR2} to see that $\phi=(1-a)\phi_\epsilon\circ q_D+a\psi$.  (This is where we use again that the matrices $A_i$ satisfy \eqref{convAstrong}, because we need all the matrices $1-e^{-r_i}E_i$ to be invertible.)
\end{proof}

\begin{cor}\label{cor5.2-1}
Suppose that we have a $k$-graph $\Lambda$ with the properties in Theorem~\ref{dominant2cpt}, and suppose  in addition that $\Lambda\backslash D$ does not have sources. Then there is a KMS$_1$ state of $(C^*(\Lambda),\alpha^r)$ if and only if $r_i=\ln\rho(A_{D,i})$ for $1\leq i\leq k$.
\end{cor}

\begin{proof} A convex combination of states of $\TC^*(\Lambda)$ factors through $C^*(\Lambda)$ if and only if all the summands do. 

As in the proof of Theorem~\ref{dominant2cpt}, we have  $r_i>\ln\rho(A_{\Lambda^0\backslash D,i})$ for all $i$.
Therefore KMS$_1$ states of $\TC^*(\Lambda\backslash D)$ have the form $\phi_\epsilon$ as described in \cite[Theorem~6.1]{aHLRS2}. Since we can recover $\epsilon$ from $\phi_\epsilon$ as 
\begin{equation}\label{gapvector}
\epsilon=\prod_{i=1}^k(1-e^{-r_i}A_{\Lambda^0\backslash D,i})m^{\phi_\epsilon},
 \end{equation} 
we deduce that the right-hand side of \eqref{gapvector} is nonzero. Since $\Lambda\setminus D$ has no sources, it follows from \cite[Proposition~4.1(b)]{aHLRS2} that $\phi_{\epsilon}$ does not factor through a state of $C^*(\Lambda\backslash D)$. So the only state of $\TC^*(\Lambda)$ which could factor through a state of $C^*(\Lambda)$ is the state $\psi$ in Proposition~\ref{extendKMS}. Thus the corollary follows from the last assertion in Proposition~\ref{extendKMS}.
\end{proof}

\begin{cor}\label{cor5.2-2}
In the situation of Theorem~\ref{dominant2cpt}, the existence of a KMS$_1$ state on $(C^*(\Lambda),\alpha^r)$ implies that $\alpha^r$ is the preferred dynamics.
\end{cor}

\begin{proof} If there is a KMS$_1$ state on $(C^*(\Lambda),\alpha^r)$, then Corollary~\ref{cor5.2-1}  implies that $r_i=\ln\rho(A_{D,i})$ for $1\leq i\leq k$.  Since each
\[
r_i\geq \ln\rho(A_i)\geq \ln\rho(A_{D,i}),
\]
we must have $r_i=\ln\rho(A_i)$ for all $i$. Thus $\alpha^r$ is the preferred dynamics.
\end{proof}

Corollary~\ref{cor5.2-2} is substantially stronger than \cite[Corollary~4.4]{aHLRS2}: there the graph is coordinatewise irreducible, and hence has only one critical component, namely $C=\Lambda^0$.

\section{Graphs with several hereditary critical components}\label{partial}

We now consider graphs in which all the critical components are hereditary, but there are more than one of them. Since there can be no paths between distinct hereditary components, none of them dominates in the sense of \S\ref{sec:dom}. The following theorem describes the KMS$_1$ states in this situation.  Recall that from \S\ref{sec:dom}
and beyond we assume that  \[\rho(A_{C, i})>1\text{\ for  $1\leq i\leq k$ and $C\in \Cc$}.\]

\begin{thm}\label{critvshered}
Suppose that $\Lambda$ is a finite $k$-graph without sinks or sources, and that all the critical components $D$ are hereditary with $\Lambda_D$ coordinatewise irreducible. We consider a dynamics $\alpha^r$ given by $r\in (0,\infty)^k$ that satisfies our standing assumption \eqref{assondyn} and has rationally independent coordinates. We write $\Ccrit$ for the set of critical components, and for each $D\in \Ccrit$ we denote by $\psi_D$ the KMS$_1$ state of $(\TC^*(\Lambda),\alpha^r)$ given by applying Proposition~\ref{extendKMS} to the component $D$. We also set
\begin{equation*}
G:=\bigcup\{D:D\in \Ccrit\}.
\end{equation*}
Then every KMS$_1$ state of $(\TC^*(\Lambda),\alpha^r)$ is a convex combination of the states \[\{\psi_D:D\in \Ccrit\}\] and a state $\phi\circ q_{G}$ lifted from a KMS$_1$ state $\phi$ of $(\TC^*(\Lambda\backslash G),\alpha^r)$. 
\end{thm}

In the proof, we adapt arguments from  the proofs of Theorem~4.3(b) in \cite{aHLRS4} and Theorem~6.5 in \cite{aHKR2}.

\begin{proof}
Suppose that $\theta$ is a KMS$_1$ state of $(\TC^*(\Lambda),\alpha^r)$. By Proposition~4.1 of \cite{aHLRS2}, we have the subinvariance relation
\begin{equation*}
A_i m^\theta\leq e^{r_i}m^\theta \quad \text{for $i\in \{1,\ldots,k\}$.}
\end{equation*}
Let $D\in \Ccrit$. By assumption $D$ is hereditary, and  thus
\begin{equation*}
A_{D,i} m^\theta|_D\leq e^{r_i}m^\theta|_D \quad \text{for  $i\in \{1,\ldots,k\}$.}
\end{equation*}
Since $D$ is critical, there exists $j_D\in \{1,\ldots,k\}$ such that 
\[
r_{j_D}=\ln\rho(A_{D,j_D})=\ln\rho(A_{j_D}),
\] 
and then
\begin{equation*}
A_{D,j_D}m^\theta|_D\leq \rho(A_{D,j_D})m^\theta|_D.
\end{equation*}
Since $\Lambda_D$ is coordinatewise irreducible,  $A_{D,j_D}$ is irreducible. Hence the subinvariance theorem \cite[Theorem~1.6]{Sen} implies that $A_{D,j_D}m^\theta|_D=\rho(A_{D,j_D})m^\theta|_D$, that is,  either $m^\theta|_D$ is a Perron--Frobenius eigenvector for $A_{D,j_D}$ or it is the zero vector. 

Suppose that $C\in \Cc\backslash\{D\}$ such that $C\Lambda D\neq\emptyset$. Then $C$ is not hereditary, and hence is not critical. Thus $\rho(A_{C, j_D})<e^{r_{j_D}}=\rho(A_{D, j_D})$,  which implies that the hypotheses of Propositions~\ref{evectors-new} and \ref{extendKMS} are satisfied. 

Set $H_D:=\{v\in \Lambda^0: v\Lambda D=\emptyset\}$ and 
\[
F_D:=\Lambda^0\backslash (D\sqcup H_D)=\{v\in \Lambda^0:v\Lambda D \neq \emptyset, v\not\in D\}.
\]
   Let $x^D$ be the unimodular Perron--Frobenius eigenvector of $A_{D,j_D}$ and set
 \begin{equation*}
y^D:= \big(\rho(A_{D,j_D})1_{F_D}-A_{F_D,j_D}\big)^{-1}A_{F_D,D,j_D}x^D.
 \end{equation*}  By Proposition~\ref{evectors-new}, relative to the decomposition $\Lambda^0=F_D\sqcup D\sqcup H_D$,   the vector $z^D=(y^D,x^D,0)$ is an eigenvector of  $A_{j_D}$ with eigenvalue $\rho(A_{D,j_D})$.   By Proposition~\ref{extendKMS}, there exists a KMS$_1$ state $\psi_D$ characterised by \eqref{defpsi}. We  define $a_D\geq 0$ by 
 \[
 m^\theta|_D=a_D\|z^D\|_1^{-1}x^D.
 \]

Our goal now is to  show that 
\begin{equation*}
\sum_{D\in\Ccrit}a_D\leq 1.
\end{equation*}
To show this, we will  prove that 
\begin{equation}
\label{sum of coefficients single vertex}
\theta(q_v)\geq \sum_{D\in \Ccrit} a_D \psi_D(q_v) \quad \text{for each $v\in \Lambda^0$};
\end{equation}
this suffices since summing over $v\in\Lambda^0$ gives 
\[
1=\theta(1)=\sum_{v\in \Lambda^0}\theta(q_v)\geq \sum_{v\in \Lambda^0}\sum_{D\in \Ccrit}a_D\psi_D(q_v)=\sum_{D\in \Ccrit}a_D.
\]
First, suppose that $v\in G$. Then $v\in D$ for some $D\in \Ccrit$ and
\[
a_D \psi_D(q_v)=a_D\|z^D\|_1^{-1}z^D_v=a_D\|z^D\|_1^{-1}x^D_v=(m^\theta|_D)_v=\theta(q_v). 
\]
All the critical components are hereditary, so for $D'\in \Ccrit\backslash \{D\}$ we have $v\in H_{D'}=\{w\in \Lambda^0:w\Lambda D'=\emptyset\}$. Thus
\[
a_{D'} \psi_{D'}(q_v)=a_{D'}\|z^{D'}\|_1^{-1}z^{D'}_v=0,
\]
and then
\begin{equation}
\label{states agree on G}
\theta(q_v)= \sum_{D\in \Ccrit} a_D \psi_D(q_v) \quad \text{for $v\in G$}
\end{equation}
which verifies \eqref{sum of coefficients single vertex} for $v\in G$.

Second, suppose that $v\not\in G$.  If $v\Lambda D=\emptyset$ for all $D\in \Ccrit$, then $v\in H_{D}$ for all $D\in \Ccrit$, and so
\[\sum_{D\in \Ccrit} a_D \psi_D(q_v)=\sum_{D\in \Ccrit} a_D\|z^D\|_1^{-1}z^D_v=0\leq \theta(q_v).\] So to verify \eqref{sum of coefficients single vertex}, it remains to consider $v\not\in G$ such that $v\Lambda D\neq \emptyset$ for at least one $D\in \Ccrit$, that is,  $v\in F_D$ for some $D$. To do this we mimic some calculations from \cite[page~2545]{aHLRS4} by looking at the paths which make a ``quick exit'' from a component $D$ in the colour $j_D$. The argument is long and complicated. 

We write $\QE_{j_D}(D)$ for  $\{\lambda\in \Lambda^{\mathbb{N}e_{j_D}}D:s(\lambda(0,d(\lambda)-e_{j_D}))\in F_D\}$, which is the set of paths  of colour $j_D$ which have source in   $D$ and such that ranges of all the edges in the path are outside $D$. We claim that 
\begin{equation}\label{orthogprojs}
\{t_\lambda t_\lambda^*:\lambda\in v\!\QE_{j_D}(D)\text{ for some $D\in \Ccrit$}\}
\end{equation}
consists of mutually orthogonal projections. To see this, fix $C,D\in \Ccrit$ and choose distinct $\lambda\in v\!\QE_{j_C}(C)$ and $\mu\in v\!\QE_{j_D}(D)$. By the Toeplitz--Cuntz--Krieger relation (T3), we have  $t_\lambda t_\lambda^*t_\mu t_\mu^*=\sum_{(\alpha,\beta)\in\Lambda^{\min}(\lambda,\mu)}t_{\lambda\alpha}t_{\mu\beta}^*$, and so it suffices to show that $\Lambda^{\min}(\lambda,\mu)=\emptyset$. 
Since $C$ and $D$ are hereditary, if there exists $(\alpha,\beta)\in \Lambda^{\min}(\lambda,\mu)$, then the common source of $\alpha$ and $\beta$ is in both $C$ and $D$, which means $C=D$. Thus $C\neq D$ implies that $\Lambda^{\min}(\lambda,\mu)=\emptyset$. So suppose that $C=D$. Since $\lambda, \mu \in \Lambda^{\mathbb{N}e_{j_D}}$, $\Lambda^{\min}(\lambda,\mu)$ will be empty unless $\lambda\in \mu\Lambda$ or $\mu\in \lambda\Lambda$. Looking for a contradiction, we assume, without loss of generality, that $\lambda\in \mu\Lambda$, say $\lambda=\mu \eta$. Since $D$ is hereditary and $s(\mu)\in D$, we see that $\eta\in \Lambda_D$. Since $\lambda\neq \mu$, we must have $d(\eta)\geq e_{j_D}$. Thus $s(\lambda(0,d(\lambda)-e_{j_D}))=s(\eta(0,d(\eta)-e_{j_D}))\in D$, which is impossible since $s(\lambda(0,d(\lambda)-e_{j_D}))\in F_D$.  Thus \eqref{orthogprojs} consists of mutually orthogonal projections as claimed,  and $q_v\geq t_\lambda t_\lambda^*$ for $\lambda\in v\Lambda$ gives
\[
\theta(q_v)\geq \sum_{D\in \Ccrit}\ \sum_{\lambda\in v\!\QE_{j_D}(D)}\theta(t_\lambda t_\lambda^*).
\]
Using the KMS condition (see  \cite[Theorem~3.1(a)]{aHLRS2}) we have
\begin{align}\label{lowerbdfortheta}
\theta(q_v)&\geq \sum_{D\in \Ccrit}\ \sum_{\lambda\in v\!\QE_{j_D}(D)}\rho(A_{D,j_D})^{-|\lambda|}\theta(q_{s(\lambda)})\\
&=\sum_{D\in \Ccrit}a_D\|z^D\|_1^{-1}\Big(\sum_{\lambda\in v\!\QE_{j_D}(D)}\rho(A_{D,j_D})^{-|\lambda|}x^D_{s(\lambda)}\Big).\notag
\end{align}
We now examine the inner sum. Since we are only looking at paths that make a quick exit from $D$, we see that
\begin{align}
\sum_{\lambda\in v\!\QE_{j_D}(D)}&\rho(A_{D,j_D})^{-|\lambda|}x^D_{s(\lambda)}\label{useQE} \\
&=\sum_{w\in D}\sum_{n=0}^\infty \rho(A_{D,j_D})^{-(n+1)}\big(A_{F_D,j_D}^nA_{F_D,D,j_D}\big)(v,w)x^D_w.\notag
\end{align}
Since all the critical components are hereditary, none of them lie in $F_D\subset \Lambda^0\backslash G$. Thus $\rho(A_{D,j_D})=e^{r_{j_D}}>\rho(A_{F_D,j_D})$, and we see that 
\begin{align*}
\eqref{useQE}&=\sum_{w\in D}\rho(A_{D,j_D})^{-1}\big(\big(1_{F_D}-\rho(A_{D,j_D})^{-1}A_{F_D,j_D}\big)^{-1}A_{F_D,D,j_D}\big)(v,w)x^D_w\\
&=\sum_{w\in D}\big((\rho(A_{D,j_D})1_{F_D}-A_{F_D,j_D})^{-1}A_{F_D,D,j_D}\big)(v,w)x^D_w\\
&=\big((\rho(A_{D,j_D})1_{F_D}-A_{F_D,j_D})^{-1}A_{F_D,D,j_D}x^D\big)_v.\\
&=y^D_v.
\end{align*}
 Since  \eqref{useQE} was the inner sum of the right-hand side of \eqref{lowerbdfortheta} and since $v\in F_D$, we get
\[
\theta(q_v)\geq\sum_{D\in\Ccrit}a_D\|z^D\|_1^{-1}y^D_v=\sum_{D\in\Ccrit}a_D\psi_D(q_v).
\]
We have now established \eqref{sum of coefficients single vertex} for all $v\in\Lambda^0$.  As mentioned above, it  follows that 
\[
\sum_{D\in\Ccrit}a_D\leq 1.
\]

Suppose that $\sum_{D\in\Ccrit}a_D=1$.  Then
\[
\begin{aligned}
\sum_{v\in \Lambda^0} \left(\theta(q_v)-\sum_{D\in \Ccrit} a_D \psi_D(q_v)\right)
&=\sum_{v\in \Lambda^0}\theta(q_v)-\sum_{D\in \Ccrit}a_D\left(\sum_{v\in \Lambda^0}\psi_D(q_v)\right)\\
&=1-\sum_{D\in \Ccrit}a_D=0.
\end{aligned}
\]
It then follows from \eqref{sum of coefficients single vertex} that $\theta(q_v)= \sum_{D\in \Ccrit} a_D \psi_D(q_v)$ for $v\in \Lambda^0$.  Since both    $\theta$ and $\sum_{D\in \Ccrit} a_D\psi_D$ are KMS$_1$ states and the coordinates of $r$ are rationally independent, it follows from \cite[Proposition~3.1(b)]{aHLRS2} that they agree on all the spanning elements  $t_\mu t_\nu^*$ of $\TC^*(\Lambda)$. Hence by linearity and continuity, we have
\[
\theta=\sum_{D\in \Ccrit} a_D \psi_D
\]
and $\theta$ is a convex combination, as required.

The other possibility is that $\sum_{D\in \Ccrit}a_D<1$. To handle this case, we adapt the argument of the last four paragraphs in the proof of \cite[Theorem~6.5]{aHKR2}. For convenience we write $E:=\Lambda^0\backslash G$.
Let 
\[
\kappa:=\Big(m^\theta-\sum_{D\in \Ccrit} a_Dm^{\psi_D}\Big)\Big|_E. 
\]
By \eqref{states agree on G},  $\theta$ and $\sum_{D\in \Ccrit} a_D\psi_D$ agree on the vertex projections $\{q_v: v\in G\}$, and a short calculation using this  shows that $\|\kappa\|_1=1-\sum_{D\in\Ccrit}a_D$. 
We claim  that  the vector
\begin{equation*}
\eta:=\prod_{i=1}^k\big(1_E-e^{-r_i}A_{E,i}\big)\kappa
\end{equation*}
belongs to $[0,\infty)^E$.  Once we have stablished that $\eta\geq 0$, we will argue that the vector $\epsilon:=\|\kappa\|_1^{-1}\eta$ gives a KMS$_1$ state $\phi_\epsilon$ of the quotient $\TC^*(\Lambda\setminus G)$ and that $\theta$ is a convex combination of $\phi_\epsilon\circ q_G$ and $\{\psi_D:D\in \Ccrit\}$.

Since $\theta$ is a KMS$_1$ state of $(\TC^*(\Lambda),\alpha^r)$, Proposition~4.1(a) of \cite{aHLRS2} implies that the vector $m^\theta=\big(\,\theta(q_v)\,\big)_{v\in \Lambda^0}$ in $[0,\infty)^{\Lambda^0}$ satisfies the subinvariance relation
\begin{equation}
\label{full sub inv product}
\prod_{i=1}^k(1_{\Lambda^0}-e^{-r_i}A_i)m^\theta\geq 0;
\end{equation}
we will show that the restriction of this vector to $E$ is    $\eta$. 

Consider the block decomposition of \eqref{full sub inv product} relative to $\Lambda^0=E\sqcup G$. 
An induction argument on the number of critical components shows that the top entry of \eqref{full sub inv product}  is
\begin{align}\label{topofsubinv}
\prod_{i=1}^k(1_E-e^{-r_i}A_{E,i})m^\theta|_E-\sum_{D\in \Ccrit}\prod_{\substack{i=1\\i\not=j_D}}^k(1_E-e^{-r_i}A_{E,i})e^{-r_{j_D}}A_{E,D,{j_D}}m^\theta|_D.
\end{align}
Now we restrict attention to  the $D$th summand of the second term: 
\begin{align*}
\prod_{\substack{i=1\\i\not=j_D}}^k&(1_E-e^{-r_i}A_{E,i})e^{-r_{j_D}}A_{E,D,{j_D}}m^\theta|_D\\
&=a_D\|z^D\|_1^{-1}\prod_{\substack{i=1\\i\not=j_D}}^k(1_E-e^{-r_i}A_{E,i})\rho(A_{D,j_D})^{-1}A_{E,D,{j_D}}x^D.
\end{align*}
Since there can be no paths from $D$ to $H_D$, nor from $F_D$ to $H_D$, relative to the decomposition $E=F_D\sqcup(H_D\backslash G)$,  this is given by
\begin{align}
a_D\|z^D\|_1^{-1}&\prod_{\substack{i=1\\i\not=j_D}}^k
\begin{pmatrix} 1_{F_D}-e^{-r_i}A_{F_D,i} & -e^{-r_i}A_{F_D, H_D\backslash G,i}\\ 0 & 1_{H_D\backslash G}-e^{-r_i}A_{H_D\backslash G,i}\end{pmatrix}
\rho(A_{D,j_D})^{-1}
\begin{pmatrix} A_{F_D,D,j_D} \\ 0 \end{pmatrix}
x^D\notag\\
&=
\begin{pmatrix} a_D\|z^D\|_1^{-1}\prod_{\substack{i=1\\i\not=j_D}}^k (1_{F_D}-e^{-r_i}A_{F_D,i})\rho(A_{D,j_D})^{-1}A_{F_D,D,j_D}x^D \\ 0 \end{pmatrix}.\label{second term decomposed}
\end{align}
Since  every critical component is hereditary, $F_D$ does not contain any critical components. Thus it follows from \eqref{assondyn} that $\rho(A_{D,j_D})=e^{r_{j_D}}>\rho(A_{F_D,j_D})$, and so the matrix $1_{F_D}-\rho(A_{D,j_D})^{-1}A_{F_D,j_D}$ is invertible. We can use this inverse to rewrite the top block of \eqref{second term decomposed} as
\begin{align*}
a_D\|z^D\|_1^{-1}&\prod_{i=1}^k(1_{F_D}-e^{-r_i}A_{F_D,i})\big(1_{F_D}-\rho(A_{D,j_D})^{-1}A_{F_D,j_D})^{-1}\rho(A_{D,j_D}\big)^{-1}A_{F_D,D,j_D}x^D\\
&=a_D\|z^D\|_1^{-1}\prod_{i=1}^k(1_{F_D}-e^{-r_i}A_{F_D,i})(\rho(A_{D,j_D})1_{F_D} -A_{F_D,j_D})^{-1}A_{F_D,D,j_D}x^D\\
&=a_D\|z^D\|_1^{-1}\prod_{i=1}^k(1_{F_D}-e^{-r_i}A_{F_D,i})y^D\\
&=\prod_{i=1}^k(1_{F_D}-e^{-r_i}A_{F_D,i})(a_Dm^{\psi_D}|_{F_D})
\end{align*}
because $m^{\psi_D}_v=\psi_D(q_v)=\|z\|_1^{-1}z^D_v$ for $v\in\Lambda^0$ and $z^D=(x^D, y^D, 0)$. 
Since $m^{\psi_D}|_{H_D\backslash G}=0$ and $E=F_D\sqcup (H_D\backslash G)$, we see that \eqref{second term decomposed} is given by
\begin{equation*}
\begin{aligned}
&\begin{pmatrix} \prod_{i=1}^k(1_{F_D}-e^{-r_i}A_{F_D,i})(a_Dm^{\psi_D}|_{F_D}) \\ 0 \end{pmatrix}\\
&\hspace{5em}=\prod_{i=1}^k
\begin{pmatrix} 1_{F_D}-e^{-r_i}A_{F_D,i} & -e^{-r_i}A_{F_D, H_D\backslash G,i}\\ 0 & 1_{H_D\backslash G}-e^{-r_i}A_{H_D\backslash G,i}\end{pmatrix}
\begin{pmatrix} a_Dm^{\psi_D}|_{F_D} \\ 0 \end{pmatrix}\\
&\hspace{5em}=\prod_{i=1}^k(1_E-e^{-r_i}A_{E,i})a_Dm^{\psi_D}|_E.
\end{aligned}
\end{equation*}
Summing over $D\in\Ccrit$ gives us back the second term in \eqref{topofsubinv}:
\[ \sum_{D\in \Ccrit}\prod_{\substack{i=1\\i\not=j_D}}^k(1_E-e^{-r_i}A_{E,i})e^{-r_{j_D}}A_{E,D,{j_D}}m^\theta|_D=\sum_{D\in\Ccrit}\prod_{i=1}^k(1_E-e^{-r_i}A_{E,i})a_Dm^{\psi_D}|_E.\]
Now, starting with the definition of $\eta$, we trace our way back to  \eqref{topofsubinv}:
\begin{align*}
\eta&= \prod_{i=1}^k(1_E-e^{-r_i}A_{E,i})\Big(m^\theta-\sum_{D\in \Ccrit}a_Dm^{\psi_D}\Big)\Big|_E
\\
&=\prod_{i=1}^k(1_E-e^{-r_i}A_{E,i}) m^\theta|_E-\prod_{i=1}^k(1_E-e^{-r_i}A_{E,i})\sum_{D\in \Ccrit}a_Dm^{\psi_D}|_E\\
&=\prod_{i=1}^k(1_E-e^{-r_i}A_{E,i}) m^\theta|_E-\sum_{D\in \Ccrit}\prod_{i=1}^k(1_E-e^{-r_i}A_{E,i})a_Dm^{\psi_D}|_E.
\end{align*}
Thus $\eta$ is the top entry of \eqref{full sub inv product}, and hence $\eta\geq 0$.

We now set 
\[
\epsilon:=\|\kappa\|_1^{-1}\eta=\Big(1-\sum_{D\in \Ccrit}a_D\Big)^{-1}\eta.
\]
Since $E$ does not contain any critical components,  by \eqref{assondyn} we have $r_i\geq\ln\rho(A_i)>\ln\rho(A_{E,i})$ for  $i\in \{1,\ldots, k\}$, and the inverse temperature $\beta=1$ is in the range for which \cite[Theorem~6.1]{aHLRS2} applies\footnote{The graph $\Lambda_E$ could have sources (see \cite[Example~8.4]{aHKR2}). So this application of the result from \cite[Theorem~6.1]{aHLRS2} depends on the observation at the start of \cite[\S8]{aHKR2} that \cite[Theorem~6.1]{aHLRS2} applies also to graphs with sources.}.
Also  $e^{r_i}>\rho(A_{E,i})$ for  $i\in \{1,\ldots, k\}$, and so all the matrices $1_E-e^{-r_i}A_{E,i}$ are invertible and we can recover $\kappa$ from $\eta$. Thus
\begin{equation*}
\begin{aligned}
\Big\|\prod_{i=1}^k (1_E-e^{-r_i}A_{E,i})^{-1}\epsilon\Big\|_1&=\Big(1-\sum_{D\in \Ccrit}a_D\Big)^{-1}\Big\|\prod_{i=1}^k (1_E-e^{-r_i}A_{E,i})^{-1}\eta\Big\|_1\\
&=\Big(1-\sum_{D\in \Ccrit}a_D\Big)^{-1}\|\kappa\|_1\\
&=1,
\end{aligned}
\end{equation*}
and it follows from \cite[Theorem~6.1(a)]{aHLRS2} that $\epsilon$ belongs to the simplex $\Sigma_1$ of that theorem for the graph $\Lambda_E=\Lambda\backslash G$. We deduce that there is a KMS$_1$ state $\phi_\epsilon$ of $(\TC^*(\Lambda\backslash G),\alpha^r)$ such that \[\phi_\epsilon(q_v)=\Big(1-\sum_{D\in \Ccrit}a_D\Big)^{-1}\kappa_v\] for all $v\in E$. Now
\[
\Big(1-\sum_{D\in \Ccrit}a_D\Big)(\phi_\epsilon\circ q_{G})+\sum_{D\in \Ccrit}a_D\psi_D
\]
is a KMS$_1$ state of $(\TC^*(\Lambda),\alpha^r)$ which agrees with $\theta$ on  the vertex projections $q_v$. Since they are both KMS$_1$ states and the coordinates of $r$ are rationally independent, it again follows from \cite[Proposition~3.1(b)]{aHLRS2} that 
\[
\theta=\Big(1-\sum_{D\in \Ccrit}a_D\Big)(\phi_\epsilon\circ q_G)+\sum_{D\in \Ccrit}a_D\psi_D,
\]
and $\theta$ is a convex combination, as required.
\end{proof}


\begin{cor}
In the situation of Theorem~\ref{critvshered}, the KMS$_1$ simplex of $(\TC^*(\Lambda),\alpha^r)$ has dimension $|\Lambda^0\backslash G|+|\Ccrit|-1$.
\end{cor}

\begin{proof}
Since all the critical components are hereditary and belong to $\Ccrit$, the dynamics on $\TC^*(\Lambda\backslash G)$ induced by $\alpha^r$ has $r_i>\ln\rho(A_{\Lambda^0\backslash G,i})$ for all $i$, and hence by \cite[Theorem~6.1]{aHLRS2} has a KMS$_1$ simplex of dimension $|\Lambda^0\backslash G|-1$. Thus the result follows from Theorem~\ref{critvshered}.
\end{proof}

\begin{cor}
Suppose that we have a $k$-graph $\Lambda$ with the properties in Theorem~\ref{critvshered}, and suppose  in addition that for $D\in\Ccrit$ the graph $\Lambda\backslash D$ does not have sources. Then a KMS$_1$ state of $\TC^*(\Lambda)$ factors through $C^*(\Lambda)$ if and only if it is a convex combination of the states \[\{\psi_D:D\in \Ccrit, r_i=\ln\rho(A_{D,i}) \text{ for }1\leq i\leq k\}.\]
\end{cor}

\begin{proof}  By Corollary~\ref{cor5.2-1}, each $\psi_D$ factors through $C^*(\Lambda)$ if and only if $r_i=\ln\rho(A_{D,i})$ for $1\leq i\leq k$.  No state $\phi_\epsilon$ of $\TC^*(\Lambda\backslash G)$  factors through $C^*(\Lambda\backslash G)$ (this follows as in the proof of Corollary~\ref{cor5.2-1}). Since a convex combination of states of $\TC^*(\Lambda)$ factors through $C^*(\Lambda)$ if and only all the summands do, the result follows. 
\end{proof}

\section{Computing the KMS$_1$ states}\label{procedure}

We now describe how to combine our results to find the extreme KMS states at the critical inverse temperature. We begin by making some general assumptions about the graphs we consider.

Throughout, we consider a finite $k$-graph $\Lambda$ with no sources and no sinks. First we assume that the graph is suitably connected:
\begin{itemize}
 \item[(A1)]
We assume that there are no trivial strongly connected components, which forces $\Lambda^0=\bigcup\{C:C\in \Cc\}$. We assume that there are no isolated subgraphs: there is no decomposition $\Lambda=\Lambda_K\sqcup \Lambda_H$ with $H\cap K=\emptyset$.
\end{itemize}
We assume that the full results of \cite{aHLRS2} apply to the reductions $\Lambda_C$:
\begin{itemize}
\item[(A2)] For all $C\in \Cc$, the graph $\Lambda_C$ is coordinatewise irreducible, and $\rho(A_{C,i})>1$ for all $i$.
\end{itemize}
Our next assumption rules out the case in which the only bridge between components  consists of edges of a single colour, as in the example in \cite[Remark~6.2]{aHKR2}.
\begin{itemize}
\item[(A3)] If $C,D\in \Cc$ and $C\Lambda^{e_j} D\not=\emptyset$ for some $j$, then $C\Lambda^{e_i}D\not=\emptyset$ for all $1\leq i\leq k$.
\end{itemize}
Since we know by (A2) that the graphs $\{\Lambda_C:C\in \Cc\}$ are coordinatewise irreducible, (A1) and (A3) imply that $\Lambda$ satisfies the hypothesis \eqref{allconnected} in Proposition~\ref{reducetocrither}. The three assumptions (A1--A3) imply that for every hereditary subset $H$ of $\Lambda^0$, the graph $\Lambda\backslash H$ obtained by removing $H$ has no sources or sinks. The examples in \cite[\S8]{aHKR2} show that otherwise sources of various types could be created.

We now consider a dynamics $\alpha^r:\RR\to \Aut \TC^*(\Lambda)$ determined by a vector $r\in (0,\infty)^k$ satisfying the standing hypothesis \eqref{assondyn} and having rationally independent coordinates $\{r_i:1\leq i\leq k\}$.
Then the following procedure will generate the extreme points of the simplex of KMS$_1$ states of the system $(\TC^*(\Lambda),\alpha^r)$. We are thinking of the components as small, even singletons, in which case $\Lambda$ is one of our favourite  dumbbell graphs.
\smallskip

\begin{itemize}
\item[(C1)] First we calculate the spectral radii of the matrices $A_{C,i}$, and identify the critical components of $\Lambda$. If any of them are not hereditary, we identify the set $H$ of Proposition~\ref{reducetocrither}. (Recall that $H$ is the the complement of the union of $\Cmcrit$ in its hereditary closure.) Then we study the system $(\TC^*(\Lambda\backslash H), \alpha^r)$. 
\smallskip

\item[(C2)] All of the critical components of $\Lambda\backslash H$  are now hereditary. For each such component $D$, we compute the common unimodular Perron--Frobenius eigenvector $x$ of the matrices $A_{D,i}$. (Theoretically, it suffices to find the eigenvector of one of them, but there is an opportunity for a reality check here.) Proposition~\ref{evectors-new} tells us how to extend $x$ to an eigenvector $z$ of $A_{\Lambda^0\backslash H, i}$, and then Proposition~\ref{extendKMS} gives us an explicit KMS$_1$ state $\psi_D$ of $(\TC^*(\Lambda\backslash H),\alpha^r)$.
\smallskip

\item[(C3)] Now we take $G$ to be the union of the critical components of $\Lambda\backslash H$. Then Theorem~\ref{critvshered} tells us that the states
\[\{\psi_D: D\in \Ccrit(\Lambda\backslash H)\}=\{\psi_D: D\in \Cmcrit(\Lambda)\}
\]
are extreme points of the KMS$_1$ simplex of $(\TC^*(\Lambda\backslash H),\alpha^r)$, and that the other KMS$_1$ states factor through the quotient map $q_G:\TC^*(\Lambda\backslash H)\to \TC^*(\Lambda\backslash (H\cup G))$. Since we have removed all the critical components from $\Lambda$, we have \[
r_i>\ln\rho(A_{\Lambda^0\backslash (H\cup G),i})\quad \text{for $1\leq i\leq k$},
\]
and Theorem~6.1 of \cite{aHLRS2} describes an explicit parametrisation $\epsilon\mapsto \phi_\epsilon$ of the KMS$_1$ simplex of $(\TC^*(\Lambda\backslash (H\cup G)),\alpha^r)$. The extreme points are those where $\epsilon$ is a multiple of a point mass $\delta_v$; in the notation of \cite[Theorem~6.1]{aHLRS2}, the multiple is $y_v^{-1}\delta_v$. (This vector $y$ is not the one of Proposition~\ref{evectors-new}.) Thus the extreme points of the simplex of KMS$_1$ states of $(\TC^*(\Lambda),\alpha^r)$ are
\begin{equation*}
\big\{\psi_D:\text{$D$ is critical for $\Lambda\backslash H$}\big\}\cup
\big\{\phi_{y_v^{-1}\delta_v}\circ q_G:v\in \Lambda^0\backslash (H\cup G)\big\}.
\end{equation*}
\end{itemize}

\section{Finding all the KMS states}\label{program}

We now describe our program for finding all the KMS states of $(\TC^*(\Lambda),\alpha^r)$. We suppose that $\Lambda$ is a finite $k$-graph with no sources and no sinks, satisfying the assumptions (A1--A3) of the previous section. We suppose that $r\in(0,\infty)^k$ has rationally independent coordinates, and has been normalised (multiplied by a suitable scalar) to ensure that the standing hypothesis \eqref{assondyn} is satisfied. Our program has three steps, which can be iterated to reduce our problem to the same problem for graphs with fewer strongly connected components. When we get down to a strongly connected graph (one with a single component), \cite[Proposition~4.2]{aHKR} implies that there is a unique KMS$_1$ state. The program consists of the following 3 items.
\begin{itemize}
\item[(P1)] For $\beta>1$, we apply \cite[Theorem~6.1]{aHLRS2}. This yields a KMS$_\beta$ simplex with $|\Lambda^0|$ extreme points  $\big\{\phi_{y_v^{-1}\delta_v}:v\in \Lambda^0\big\}$. 
\smallskip

\item[(P2)] At $\beta=1$, we apply the procedure of \S\ref{procedure}. Let $H$ be the hereditary set  described in Proposition~\ref{reducetocrither} and let $G$ be the union of the critical components of $\Lambda\backslash H$. Then we follow the steps (C1--C3), arriving at a KMS$_1$ simplex with extreme points parametrised by the critical components in $\Lambda\backslash H$ and the vertices $v$ in $\Lambda^0\backslash(G\cup H)$.
\smallskip

\item[(P3)] For $\beta<1$, we start with a lemma.  

\begin{lem}\label{lemP3} Every KMS$_\beta$ state of $(\TC^*(\Lambda),\alpha^r)$ factors through a KMS$_\beta$ state of $(\TC^*(\Lambda\backslash(H\cup G)),\alpha^r)$. 
\end{lem}
\begin{proof} Let $\phi$ be a KMS$_\beta$ state of $(\TC^*(\Lambda),\alpha^r)$.  Fix $v\in H\cup G$.  We will show that $\phi(q_v)=0$. Recall that the critical components of $\Lambda\setminus H$ are also critical components of $\Lambda$ (in fact the minimal ones). So there exists a critical component $C$ of $\Lambda$ such that $C\Lambda v\neq \emptyset$.  We restrict $\phi$ to get a KMS$_\beta$ state of $(\TC^*(\Lambda_C),\alpha^r)$.  By \cite[Proposition~4.1(a)]{aHLRS2}
\[
A_{C,i}m^\phi|_C\leq e^{\beta r_i}m^\phi|_C
\]
for $1\leq i\leq k$. Since each $A_{C,i}$ is irreducible,  the subinvariance theorem \cite[Theorem~1.6]{Sen} gives that
\[
m^\phi|_C\neq 0\Longrightarrow \rho(A_{C, i})\leq e^{\beta r_i}\Longrightarrow r_i^{-1}\ln\rho(A_{C, i})\leq \beta<1
\]
for all $i$.  But $C$ is $j$-critical, say, and this gives $r_j^{-1}\ln\rho(A_{C, j})=1$, which is impossible. Thus $m^\phi|_C= 0$.

Now let $w\in C$ and choose $n\in\NN^k$ such that $w\Lambda^nv\neq \emptyset$. Then
\[
0=\phi(q_w)\geq \sum_{\lambda\in w\Lambda^n}\phi(t_\lambda t_\lambda^*)=\sum_{\lambda\in w\Lambda^n}e^{-\beta r\cdot n}\phi(t_{s(\lambda)})
\]
using that $\phi$ is a KMS state.
Thus  $\phi(t_{s(\lambda)})=0$ for all $\lambda\in w\Lambda^n$. In particular, $\phi(q_v)=0$.  It follows from \cite[Lemma~6.2]{AaHR} that $\phi$ vanishes on the ideal $I_{H\cup G}$, and hence $\phi$ factors through a KMS$_\beta$ state of $(\TC^*(\Lambda\backslash(H\cup G)),\alpha^r)$. 
\end{proof}

Next we need to check that the graph $\Sigma:=\Lambda\backslash(G\cup H)$ satisfies the assumptions (A1--A3). 
 Since $H$ is hereditary, every component $C$ with $C\cap H\not=\emptyset$ lies entirely inside $H$. So $H$ is a union of strongly connected components, all of which are nontrivial by (A1). The set $G$ is the union of the critical components in $\Lambda\backslash H$. So $G\cup H$ is a union of strongly connected components of $\Lambda^0$. Since $\Lambda$ has no trivial components and all components are contained in one of $\Lambda^0\backslash(G\cup H)$ or $G\cup H$, $\Lambda^0\backslash(G\cup H)$ is also the union of the components it contains. So $\Sigma$ has no nontrivial components. But $\Sigma$ may not satisfy the second part of (A1): there may be disjoint subsets $K_j$ of $\Sigma^0$ which are unions of components, which satisfy $\Sigma^0=\bigcup_{j=1}^nK_j$, and which do not speak to each other in the sense that $K_j\Sigma K_l=\emptyset$ for $j\not=l$.\footnote{For example, suppose that $\Lambda$ has three components arranged as in \S\ref{sec3cpts}, that the hereditary component $D$ belongs to $\Cmcrit$, and and that $A_{C,B,i}=0$ for all $i$. Then $\Lambda\backslash H$ is the disjoint union of $\Lambda_C$ and $\Lambda_B$.} 
 
 We set $P_j=\sum_{v\in K_j}q_v$, and 
observe that \[\{q_v:v\in K_j\}\cup\{t_\lambda:\lambda\in K_j\Sigma K_j\}\] is a Toeplitz--Cuntz--Krieger $\Sigma_{K_j}$-family in $P_j\TC^*(\Sigma)P_j$ which gives an isomorphism of $\TC^*(\Sigma_{K_j})$ onto $P_j\TC^*(\Sigma)P_j$. Since $P_jP_k=0$ for $j\not=k$, we have \[\TC^*(\Sigma)=\bigoplus_{j=1}^n  P_j\TC^*(\Sigma)P_j.\] Thus the KMS$_\beta$ states of $(\TC^*(\Sigma),\alpha^r)$ are convex combinations of KMS$_\beta$ states of $(\TC^*(\Sigma_{K_j}),\alpha^r)$. Since each $\Sigma_{K_j}$ is smaller than $\Lambda$, we have reduced the problem of computing KMS$_\beta$ states of $(\TC^*(\Lambda),\alpha^r)$ to the analogous problem for smaller graphs, each of which satisfies the hypotheses (A1--A3). We now study one of these smaller graphs, $\Sigma_K$, say.

Since we removed all the critical components in $\Lambda\backslash H$ when we removed $G$,  
\[
\beta_c:=\max\big\{r_i^{-1}\ln\rho(A_{\Sigma_K,i}):1\leq i\leq k\big\}
\]
is strictly less than $1$. This is another critical inverse temperature for the original system $(\TC^*(\Lambda),\alpha^r)$. To make the results of \S\ref{critcpts}--\S\ref{procedure} available verbatim, we consider the dynamics $\alpha^{\beta_cr}$. This new dynamics satisfies the standing hypothesis \eqref{assondyn} for $\Lambda_{\Sigma_K}$, and has the same KMS states as $(\TC^*(\Sigma_K),\alpha^r)$: the KMS$_\beta$ states of $(\TC^*(\Sigma_K),\alpha^r)$ are the KMS$_{\beta_c^{-1}\beta}$ states of $(\TC^*(\Sigma_K),\alpha^{\beta_cr})$ \cite[Lemma~2.1]{aHKR}. Now the system $(\TC^*(\Sigma_K),\alpha^{\beta_cr})$ is one to which we can apply our program (P1--P3). 
\end{itemize}

Since the requirement \eqref{assondyn} implies that $\Lambda$ has at least one critical component, the set $G$ contains at least one component. Thus the graphs $\Sigma_{K_j}$ have strictly fewer strongly connected components than $\Lambda$, and the iterative process we have described must terminate after finitely many steps.

\section{Applications and examples}\label{apps}

We now discuss implementation of our program. In the various subsections, we focus on graphs with relatively few components (\S\ref{sec2cpts} and \S\ref{sec3cpts}), graphs with just one colour, where we carry out a reality check by comparing with the results for ordinary graph algebras in \cite{aHLRS3}, and graphs in which the components are singletons, where we can do specific calculations like those for $1$-graphs in \cite{aHRabel}. We are reassured that our program does not apparently run into new difficulties.

\subsection{Graphs with two components}\label{sec2cpts} We first apply our program to a graph $\Lambda$ with exactly two nontrivial strongly connected components. Our assumptions (A1--A3) imply that one component $C$ is forwards hereditary, the other component $D$ is hereditary, and the vertex matrices of $\Lambda$ have the form
\[
A_i=\begin{pmatrix}A_{C,i}&A_{C,D,i}\\0&A_{D,i}
\end{pmatrix}
\quad\text{with $A_{C,D,i}\not=0$ for all $i$.}
\]

For $\beta>1$, our first step (P1) gives a simplex of KMS$_\beta$ states with $|\Lambda^0|=|C|+|D|$ extreme points.

At $\beta=1$, (P2) tells us to apply the procedure (C1--C3) of \S\ref{procedure}. First, suppose that $C$ is critical. Then $\Ccrit$ is $\{C\}$ or $\{C,D\}$. Either way, only $C$ is minimal. Thus $H=D$, $\Lambda\backslash H=\Lambda_C$ and Proposition~\ref{reducetocrither} implies that every KMS$_1$ state of $(\TC^*(\Lambda),\alpha^r)$ factors through a KMS$_1$ state of $(\TC^*(\Lambda_C),\alpha^r)$. Since $\Lambda_C$ is coordinatewise irreducible and $\{r_i:1\leq i\leq k\}$ are rationally independent, there is a unique KMS$_1$ state on $\TC^*(\Lambda_C)$ by \cite[Theorem~4.2]{aHKR}.  Thus there is a unique KMS$_1$ state on $\TC^*(\Lambda)$ as well.

Second, suppose that $C$ is not critical. Then $\Ccrit=\{D\}$, $H=\emptyset$, and $\Lambda\backslash(G\cup H)=\Lambda_C$. Thus every KMS$_1$ state of $\TC^*(\Lambda)$ is a convex combination of $\psi_D$ and a KMS$_1$ state of $(\TC^*(\Lambda_C),\alpha^r)$. Since $C$ is not critical, $r_i>\ln\rho(A_{C,i})$ for $1\leq i\leq k$, and \cite[Theorem~6.1]{aHLRS2} gives $|\Lambda^0_C|=|C|$ extreme KMS$_1$ states of $(\TC^*(\Lambda_C),\alpha^r)$. Thus the KMS$_1$-simplex of $\TC^*(\Lambda)$ has $|C|+1$ extreme points. 

For $0<\beta<1$, we follow (P3). If $C$ is critical, then $H=D$, $G=C$ and $\Sigma=\Lambda\backslash (G\cup H)$ is empty, and there are no KMS$_\beta$ states. So we suppose that $C$ is not critical. Then
\[
\beta_c:=\max\{r_i^{-1}\ln\rho(A_{C,i}):1\leq i\leq k\}
\]
is strictly less than $1$, and  (P3) tells us to apply the program to $\Lambda_C$. For $\beta_c<\beta<1$, the KMS$_\beta$ states are lifted from KMS$_\beta$ states of $(\TC^*(\Lambda_C),\alpha^r)$, and (P1) gives us a KMS$_\beta$ simplex with $|C|$ extreme points. Then (P2) gives us a single KMS$_{\beta_c}$ state of $(\TC^*(\Lambda_C),\alpha^r)$, and hence also of the original system. Now (P3) tells us to look at $C\backslash C=\emptyset$, and the original system has no KMS$_\beta$ states for $\beta<\beta_c$.

\subsection{Graphs with three components}\label{sec3cpts}
We now consider a finite $k$-graph $\Lambda$ which satisfies the assumptions (A1--A3) of \S\ref{procedure}, and which has three strongly connected components.  Recall from \cite[Proposition~3.1]{aHKR2} that we can order the components so that the vertex matrices $A_i$ are simultaneously block upper-triangular. The assumption of ``no trivial components'' says that these decompositions have no strictly upper-triangular diagonal blocks. The component $C$ such that the $\{A_{C,i}\}$ are the top blocks is forwards hereditary, and the component $D$ such that the $\{A_{D,i}\}$ are the bottom blocks is hereditary. We call the remaining component $B$. Then each $A_i$ has the form
\begin{equation}\label{block3cpts}
A_i=\begin{pmatrix}A_{C,i}&A_{C,B,i}&A_{C,D,i}\\0&A_{B,i}&A_{B,D,i}\\0&0&A_{D,i}\end{pmatrix}.
\end{equation}

We now consider a dynamics $\alpha^r:\TT\to \Aut\TC^*(\Lambda)$ such that $r\in (0,\infty)^k$ satisfies \eqref{assondyn} and has rationally independent coordinates. We want to find the KMS$_1$ states of the system $(\TC^*(\Lambda),\alpha^r)$, and we run through the program. The first step (P1) gives a KMS$_\beta$ simplex with $|\Lambda^0|$ extreme points for every $\beta>1$.

Next (P2) tells us to look at the KMS$_1$ states using the procedure (C1--C3) of \S\ref{procedure}. Suppose first that $C$ is critical. Since (A1) says that $C$ and $D$ are not isolated, there must be paths from $D$ to $C$ (possibly going through $B$). Thus the set $H$ in (C1) is either $D$ or $B\cup D$. Either way, (C1) reduces the problem of computing the KMS$_1$ states on $\TC^*(\Lambda)$ to the analogous problem for a graph with one or two components, and the analysis of \S\ref{sec2cpts} tells us how to do this.

We suppose next that $C$ is not critical and  $B$ is critical. If there exists $j$ such that the block $A_{B,D,j}$ in \eqref{block3cpts} is nonzero, then the set $H$ in (C1) is $D$, and again the problem reduces to the same one for the graph $\Lambda_{C\cup B}$ with two components. This is either a disjoint union of two irreducible graphs, in which case $\TC^*(\Lambda_{C\cup B})=\TC^*(\Lambda_C)\oplus\TC^*(\Lambda_B)$ and we can study the summands separately, or the analysis of \S\ref{sec2cpts} applies. So we suppose that $A_{B,D,i}=0$ for all $i$. Then both $B$ and $D$ are hereditary. If $D$ is also critical, then (C2) gives two KMS$_1$ states $\psi_B$ and $\psi_D$, and a $(|C|-1)$-dimensional simplex of KMS$_1$ states lifted from states of $(\TC^*(\Lambda_C),\alpha^r)$; thus the KMS$_1$ simplex of $(\TC^*(\Lambda),\alpha^r)$ has dimension $|C|+1$. If $D$ is not critical, the simplex has dimension $|C|+|D|$.

Finally, we suppose that neither $C$ nor $B$ is critical. Then $D$ has to be critical, and the KMS$_1$ simplex of $(\TC^*(\Lambda),\alpha^r)$ has dimension $|C|+|B|$.

Below $\beta=1$, step (P3) tells us to run the program for a smaller graph $\Sigma:=\Lambda\backslash(G\cup H)$. Since $\Sigma$ has one or two components, we have seen in the previous subsection that the program will give us all the KMS states.

\subsection{Comparison with previous results for $1$-graphs}

When $k=1$, a $k$-graph $\Lambda$ is the path category of a directed graph $E$, and its Toeplitz algebra is the algebra $\TC^*(E)$ whose KMS states were analysed in \cite{aHLRS1} and \cite{aHLRS4}. We write $A$ for the vertex matrix of $E$, and consider the dynamics $\alpha:t\mapsto \gamma_{e^{it}}$ on $\TC^*(E)$ studied in \cite{aHLRS1, aHLRS4}. Then the dynamics in this paper is given by $\alpha':t\mapsto \alpha_{t\ln\rho(A)}$; thus the KMS$_\beta$ states of $(\TC^*(E),\alpha)$ are the KMS$_{\beta (\ln\rho(A))^{-1}}$ states of $(\TC^*(E),\alpha')$ (see, for example, \cite[Lemma~2.1]{aHKR}). In particular, the KMS$_{\ln\rho(A)}$ simplex of $(\TC^*(E),\alpha)$ should be the KMS$_1$ simplex of our $(\TC^*(E),\alpha')$.

To find the KMS$_{\ln\rho(A)}$ states of $(\TC^*(E),\alpha)$, we apply the procedure of \cite[Theorem~4.3]{aHLRS4}, which focuses on the set $\mc(E)$ of ``minimal critical components'' in $E^0/\!\!\sim$. Under our hypotheses, $E^0/\!\!\sim$ is the set $\Cc$ of nontrivial strongly connected components of $\Lambda^0=E^0$; a component $C$ is critical if $\rho(A_C)=\rho(A)$, and minimal if $D\leq C$ and $D$ critical imply $D=C$. Part~(a) of \cite[Theorem~4.3]{aHLRS4} describes KMS$_{\ln\rho(A)}$ states $\{\psi_C:C\in \mc(E)\}$, and part (b) says that every KMS$_{\ln\rho(A)}$ state is a convex combination of the $\psi_C$ and a state lifted from the quotient associated to the hereditary closure of $C':=\bigcup_{C\in \mc(E)}C$. When we carry out our procedure from \S\ref{procedure}, we take two quotients: first in step (C1) by an ideal $I_H$, and then in step (C3) by an ideal $I_G$. The hereditary closure of $C'$ is precisely $G\cup H$, and hence \cite[Theorem~4.3]{aHLRS4} merely does both quotients in one hit. Thus the KMS$_{\ln\rho(A)}$ states of $(\TC^*(E),\alpha)$ and the KMS$_1$ states of our $(\TC^*(E),\alpha')$ are the same, as expected. (Which is reassuring, because our constructions in \S\ref{extendPF}--\S\ref{partial} were based on those of \cite[\S4]{aHLRS4}.)

The other main result in \cite{aHLRS4} describes the KMS$_\beta$ states at a fixed inverse temperature $\beta$, as follows. First, consider the hereditary closure $H_\beta$ of the components $C$ with $\ln \rho(A_C)>\beta$. If $H_\beta=E^0$, there are no KMS$_\beta$ states. Otherwise,  Theorem~5.3 of \cite{aHLRS4} says that all KMS$_\beta$ states of $(\TC^*(E),\alpha)$ factor through a state of $(\TC^*(E\backslash H_\beta),\alpha)$, and are then given by \cite[Theorem~3.1]{aHLRS1} provided that $\beta>\ln\rho(A_{E^0\backslash H_\beta})$, that is, $\beta$ is not critical. None of these states factor through $C^*(E)$. (The subtleties in \cite[Theorem~5.3]{aHLRS4} involving the saturation $\Sigma H_\beta$ do not arise here because we are not allowing trivial components.)

If $\beta$ is critical, so that $\beta=\ln\rho(A_C)$ for some component $C$, then \cite[Theorem~5.3(c)]{aHLRS4} tells us to look also at the hereditary closure $K_\beta$ of $\{C:\ln\rho(A_C)\geq \beta\}$, which strictly contains $H_\beta$. Then applying \cite[Theorem~4.3]{aHLRS4} to $E\backslash H_\beta$ gives the extreme KMS$_\beta$ states $\{\psi_C:C\in \mc(E\backslash H_\beta)\}$ and other states lifted from $\TC^*(E\backslash K_\beta)$; this is our step (C3). In our situation (no trivial components), only convex combinations of the $\psi_C$ factor through states of $C^*(E)$.

\subsection{Concrete examples}\label{examples}

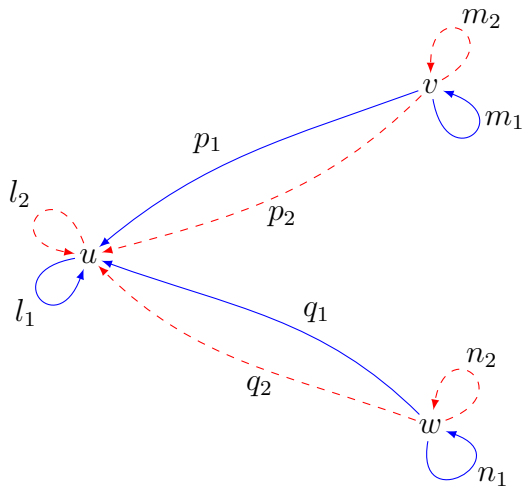
\begin{figure}[h]
\begin{tikzpicture}[scale=1.5]
 \node[inner sep=0.5pt, circle] (u) at (0,0) {$u$};
    \node[inner sep=0.5pt, circle] (v) at (3,1.5) {$v$};
    \node[inner sep=0.5pt, circle] (w) at (3,-1.5) {$w$};
\draw[-latex, blue] (u) edge [out=190, in=250, loop, min distance=20, looseness=2.5] (u);
\draw[-latex, red, dashed] (u) edge [out=110, in=170, loop, min distance=20, looseness=2.5] (u);
\draw[-latex, blue] (v) edge [out=280, in=340, loop, min distance=20, looseness=2.5] (v);
\draw[-latex, red, dashed] (v) edge [out=30, in=90, loop, min distance=20, looseness=2.5] (v);
\draw[-latex, blue] (w) edge [out=260, in=340, loop, min distance=20, looseness=2.5] (w);
\draw[-latex, red, dashed] (w) edge [out=20, in=80, loop, min distance=20, looseness=2.5] (w);

\draw[-latex, red, dashed] (v) edge [out=225, in=15]  (u);
\draw[-latex, blue] (v) edge [out=200, in=40]   (u);
\draw[-latex, red, dashed] (w) edge [out=160, in=315] (u);
\draw[-latex, blue] (w) edge [out=135, in=340]  (u);
\node at (-0.55, -0.5) {$l_1$};
\node at (-0.6,0.55) {$l_2$};
\node at (3.65, 1.2) {\color{black} $m_1$}; 
\node at (3.45, 2.1) {\color{black} $m_2$};
\node at (3.55, -1.95) {\color{black} $n_1$};
\node at (3.45, -0.9 ) {\color{black} $n_2$};
\node at (1.05, 1) {$p_1$}; 
\node at (1.7, 0.35) {$p_2$};
\node at (2,-0.5) {${q_1}$};
\node at (1.5, -1.15) {$q_2$};
\end{tikzpicture}
\caption{A dumbbell graph with $2$ hereditary components.}\label{fig3cptdumbbell}
\end{figure}

Our examples are dumbbell graphs with three components and minimal activity between the components. The most interesting case seems to be graphs with skeleton shown in Figure~\ref{fig3cptdumbbell}, where as usual the blue loop at $w$ labelled $n_1$ means there are $n_1$ blue loops at $w$. Thus the vertex matrices of the $2$-coloured graph in Figure~\ref{fig3cptdumbbell} for the ordering $\{u,v,w\}$ of $\Lambda^0$ are 
\[
A_1=\begin{pmatrix}l_1&p_1&q_1\\0&m_1&0\\0&0&n_1\end{pmatrix}\quad\text{and}\quad 
A_2=\begin{pmatrix}l_2&p_2&q_2\\0&m_2&0\\0&0&n_2\end{pmatrix}.
\]
Kumjian and Pask proved that the graph in Figure~\ref{fig3cptdumbbell} is the skeleton of a $2$-graph if and only if the matrices $A_1$ and $A_2$ commute \cite[\S6]{KP} (and is then the skeleton of many such graphs). For the matrices $A_1$ and $A_2$, this is equivalent to 
\begin{align}\label{equivtocomm}
l_1p_2+p_1m_2&=l_2p_1+p_2m_1,\text{ and }\\
l_1q_2+q_1n_2&=l_2q_1+q_2n_1.\notag
\end{align} 
 
\begin{example}
We consider a $2$-graph $\Lambda$ with skeleton shown in Figure~\ref{fig3cptdumbbell}, so that the  
vertex matrices are
\[
A_1=\begin{pmatrix}2&2&3\\0&4&0\\0&0&5\end{pmatrix}\quad\text{and}\quad
A_2=\begin{pmatrix}2&1&2\\0&3&0\\0&0&4\end{pmatrix}.
\]
We consider the preferred dynamics given by $r=(\ln 5, \ln 4)$. It follows from \cite[Proposition~A.1]{aHKR2} that $\ln 4$ and $\ln 5$ are rationally independent.

In this graph, there is one critical component $D=\{w\}$. The other blocks in the decomposition of $\Lambda^0$ in Proposition~\ref{evectors-new} are $F=\{u\}$ and $H=\{v\}$, and for the ordering $\{u,w,v\}$ of $\Lambda^0$, the vertex matrices become
\[
A_1=\begin{pmatrix}2&3&2\\0&5&0\\0&0&4\end{pmatrix}\quad\text{and}\quad
A_2=\begin{pmatrix}2&2&1\\0&4&0\\0&0&3\end{pmatrix}.
\]
In the notation of Proposition~\ref{evectors-new}  we have $E_1=(2)=E_2$, $B_1=(3)$ and $B_2=(2)$. The common unimodular Perron--Frobenius eigenvector $x$ of $A_{D,1}=(5)$ and $A_{D,2}=(4)$ is the scalar $1$, and the vector $y$ in Proposition~\ref{evectors-new} is the scalar 
\[
(\rho(A_{D,1})1_F-E_1)^{-1}B_1x=(5-2)^{-1}3=1.
\] 
(As a reality check, we confirm that $(\rho(A_{D,2})1_F-E_2)^{-1}B_2x=(4-2)^{-1}2$ is also $1$.) Thus the vector $z$ in Proposition~\ref{extendKMS}  is $(1,1,0)$, and that proposition gives a KMS$_1$ state $\psi_D$ of $(\TC^*(\Lambda),\alpha^r)$ such that $\psi_D(q_u)=2^{-1}=\psi_D(q_w)$ and $\psi_D(q_v)=0$. Since $\rho(A_{D,1})=5=\rho(A_1)$ and $\rho(A_{D,2})=4=\rho(A_2)$, and the spectral radii of the other diagonal blocks are all smaller, Theorem~\ref{critvshered} says that every KMS$_1$ state of $(\TC^*(\Lambda),\alpha^r)$ is a convex combination $a\psi_D+(1-a)(\theta\circ q_D)$  for some KMS$_1$ state $\theta$  of $(\TC^*(\Lambda\backslash D),\alpha^r)=(\TC^*(\Lambda_{F\cup H}),\alpha^r)$. 

So we want to analyse the KMS$_1$ states on $\TC^*(\Lambda_{F\cup H})$ for the dynamics induced by the preferred dynamics $\alpha^r$ on $\TC^*(\Lambda)$. 
The graph $\Lambda_{F\cup H}$ has vertex matrices 
\[
A_{F\cup H,1} =\begin{pmatrix}2&2\\0&4\end{pmatrix}\quad\text{and}\quad
A_{F\cup H,2}=\begin{pmatrix}2&1\\0&3\end{pmatrix}.
\]
Since the dynamics satisfies $r_i=\ln\rho(A_i)>\ln\rho(A_{F\cup H,i})$ for $i=1,2$, the inverse temperature $\beta=1$ lies in the range for which Theorem~6.1 of \cite{aHLRS2} applies to $(\TC^*(\Lambda_{F\cup H}),\alpha^r)$. Thus every KMS$_1$ state of $(\TC^*(\Lambda_{F\cup H}),\alpha^r)$ has the form $\phi_\epsilon$ for $\epsilon$ in a simplex lying in $[0,\infty)^{F\cup H}$, which is described in \cite[Theorem~6.1(c)]{aHLRS2}.

 As we did for other dumbbell graphs in \cite[\S4]{aHRabel}, we directly solve the subinvariance relation
\begin{equation}\label{subinv}
\epsilon:=\prod_{i=1}^2(1-e^{-r_i}A_{F\cup H,i})m\geq 0
\end{equation}
for a vector $m\in [0,\infty)^{F\cup H}$ satisfying $\|m\|_1=1$. 
Thus we seek $t\in [0,1]$ such that $m=(m_u, m_v)=(1-t,t)$ satisfies \eqref{subinv}. Multiplying out the product gives 
\[
\epsilon=\frac{1}{20}\Big(5\begin{pmatrix}1&0\\0&1\end{pmatrix}-A_{F\cup H,1}\Big)\Big(4\begin{pmatrix}1&0\\0&1\end{pmatrix}-A_{F\cup H,2}\Big)\begin{pmatrix}1-t\\t\end{pmatrix}
=\frac{1}{20}\begin{pmatrix}6-11t\\t\end{pmatrix},
\]
which is nonnegative if and only if $t\in [0,\frac{6}{11}]$. Taking $t=\frac{6}{11}$ gives a KMS$_1$ state $\psi_H$ of $(\TC^*(\Lambda_{F\cup H}),\alpha^r)$ such that  $\psi_H(q_u)=\frac{5}{11}$ and $\psi_H(q_v)=\frac{6}{11}$; taking $t=0$ gives a state $\psi$ such that $\psi(q_v)=0$ and $\psi(q_u)=1$, which by an application of \cite[Lemma~6.2]{AaHR} factors through a state $\phi$ of $\TC^*(\Lambda_F)$. 

Thus the KMS$_1$ simplex of $(\TC^*(\Lambda),\alpha^r)$ has extreme points $\psi_D$, $\psi_H\circ q_D$ and $\phi\circ q_{D\cup H}$.

We now consider $\beta<1$. The only critical component of $\Lambda$ is $D=\{w\}$, so (P3) tells us to apply (P1) or (P2) to $\TC^*(\Lambda_{F\cup H})$, and we get KMS$_\beta$ states for $\beta<1$ which factor through the quotient map $q_D$.  The next critical inverse temperature is 
\[
\beta_c=\max\Big\{\frac{\ln 4}{\ln 5},\ \frac{\ln 3}{\ln 4}\Big\}=\frac{\ln 4}{\ln 5}.
\]
For $\beta_c<\beta<1$, Theorem~6.1 of \cite{aHLRS2} gives a KMS$_\beta$ simplex of dimension $1$. Since $r$ has rationally independent coordinates\footnote{There is a subtlety here. Strictly speaking, we are applying Theorem~\ref{dominant2cpt} to the dynamics $\alpha^{\beta_cr}$ associated to the vector $\beta_cr=\frac{\ln 4}{\ln 5}r$. But this vector has rationally independent coordinates if and only if $r$ does, and $r_1=\ln 5$ and $r_2=\ln 4$ are rationally independent by \cite[Proposition~A.1]{aHKR2}.}, Theorem~\ref{dominant2cpt} says that there is a unique KMS$_{\beta_c}$  state $\psi_{\{v\}}=\psi_H$ of $(\TC^*(\Lambda_{F\cup H},\alpha^r)$ from the critical components, and hence again $1$-dimensional simplices of KMS$_{\beta_c}$ states on $(\TC^*(\Lambda\backslash D),\alpha^r)$ and $(\TC^*(\Lambda),\alpha^r)$. For $\beta<\beta_c$, all KMS$_\beta$ states factor through $q_{\{v,w\}}=q_{D\cup H}$. The dynamics on $\TC^*(\Lambda_{F})$ has another critical inverse temperature at 
\[
\beta_c'=\max\Big\{\frac{\ln 2}{\ln 5},\ \frac{\ln 2}{\ln 4}\Big\}=\frac{\ln 2}{\ln 4},
\]
and a single KMS$_\beta$ state for $\beta_c'\leq \beta<\beta_c$. Thus so does $(\TC^*(\Lambda),\alpha^r)$. For $\beta<\beta_c'$, there are no KMS$_\beta$ states.
\end{example}

In the previous example, we had $\rho(A_{\{u\},i})<\rho(A_{\{v\},i})<\rho(A_{\{w\},i})$ for all~$i$. In the next example there are also two hereditary components, but neither dominates the other. This is where our new Theorem~\ref{critvshered} is useful.

\begin{example}
We take a $2$-graph $\Lambda$ with skeleton as described in Figure~\ref{fig3cptdumbbell}, with $p=(1,2)$ and $q=(1,1)$. Then one checks that $l=(5,3)$, $m=(10,13)$ and $n=(11,9)$ satisfy the relations \eqref{equivtocomm}, and hence for these choices there is a $2$-graph $\Lambda$ with the skeleton in Figure~\ref{fig3cptdumbbell}. With respect to the ordering $\{u,v,w\}$ of $\Lambda^0$, the vertex matrices are
\begin{equation*}
A_1=\begin{pmatrix} 5&1&1\\0&10&0\\0&0&11 \end{pmatrix}\quad\text{and}\quad A_2=\begin{pmatrix} 3&2&1\\0&13&0\\0&0&9 \end{pmatrix}.
\end{equation*}
The numbers have been chosen quite carefully: the matrices have to commute, and we have chosen numbers so each of $l,m,n$ has coordinates that are coprime, so the vertex matrices of the subgraphs $\Lambda_C$ have $\{\ln\rho(A_{C,i}):1\leq i\leq k\}$ rationally independent for all components $C$ (see \cite[Proposition~A.1]{aHKR2}). We also have $\rho(A_1)=11$ and $\rho(A_2)=13$, so $\ln\rho(A_1)$ and $\ln\rho(A_2)$ are rationally independent too. Let $C:=\{u\}$, $B:=\{v\}$ and $D:=\{w\}$. 
Then we have 
\begin{align*}
\rho(A_{B,1})&<\rho(A_1)=11=\rho(A_{D,1}),\\ 
\rho(A_{D,2})&<\rho(A_2)=13=\rho(A_{B,2}),\text{ and}\\
\rho(A_{C,i})&\leq 5<9\leq\min\big\{\rho(A_{B,i}),\rho(A_{D,i})\big\}\quad\text{for $i=1,2$.}
\end{align*}
We consider the preferred dynamics $\alpha^r$ on $\TC^*(\Lambda)$, so that $r=(\ln 11, \ln 13)$.

Proposition~\ref{extendKMS} gives two KMS$_1$ states $\psi_B$ and $\psi_D$ of $(\TC^*(\Lambda),\alpha^r)$ such that
\[
m^{\psi_B}=\begin{pmatrix}\psi_B(q_u)\\\psi_B(q_v)\\\psi_B(q_w)\end{pmatrix}
=\frac{5}{6}\begin{pmatrix}1/5\\1\\0\end{pmatrix}\quad\text{and}\quad
m^{\psi_D}
=\frac{6}{7}\begin{pmatrix}1/6\\0\\1\end{pmatrix}.
\]
For all $e\in v\Lambda^{e_1}$ we have $s(e)=v$, and hence Proposition~\ref{extendKMS} implies that
\[
\textstyle{\psi_B(t_et_e^*)=\rho(A_1)^{-1}\psi_B(q_v)=(11)^{-1}\cdot\frac{5}{6}=\frac{5}{66}.}
\]
Thus
\[
\psi_B\Big(\sum_{e\in v\Lambda^{e_1}}t_et_e^*\Big)=10\cdot\textstyle{\frac{5}{66}}<\textstyle{\frac{5}{6}}=\psi_B(q_v),
\]
and the state $\psi_B$ does not factor through a state of $C^*(\Lambda)$. A similar argument (using red edges instead of blue ones) shows that $\psi_D$ does not factor through $C^*(\Lambda)$ either.

The dynamics on the quotient $q_{B\cup D}(\TC^*(\Lambda))=\TC^*(\Lambda_C)$ induced by the preferred dynamics $\alpha^r$ on $\TC^*(\Lambda)$ falls in the range covered by \cite[Theorem~6.1]{aHLRS2}. Hence there is a unique KMS$_1$ state $\phi$ on $(\TC^*(\Lambda_C),\alpha^r)$. Since both components $B$ and $D$ are critical, Theorem~\ref{critvshered} implies that the KMS$_1$ simplex of $(\TC^*(\Lambda),\alpha^r)$ has extreme points $\psi_B$, $\psi_D$ and $\phi\circ q_{B\cup D}$.

For $\beta<1$, Lemma~\ref{lemP3} implies that all KMS$_\beta$ states of $\TC^*(\Lambda)$ factor through $\TC^*(\Lambda_C)$.
The induced dynamics $\alpha^r$ on $\TC^*(\Lambda_C)$ has critical inverse temperature 
\[
\beta_c=\max\Big\{\frac{\ln 5}{\ln 11},\ \frac{\ln 3}{\ln 13}\Big\}=\frac{\ln 5}{\ln 11}.
\]
Since $11$ and $13$ are coprime, the vector $r$ has rationally independent coordinates, and it follows as in the previous example from \cite[Theorem~6.1]{aHLRS2} and Theorem~\ref{dominant2cpt} (or \cite[Proposition~4.2]{aHKR}) that the system $(\TC^*(\Lambda),\alpha^r)$ has a single KMS$_\beta$ state for $\beta_c\leq \beta<1$. Also, there are no KMS$_\beta$ states of $\TC^*(\Lambda)$ for $\beta<\beta_c$.
\end{example}

\appendix
\section{Dumbbell graphs with three components}\label{app3vertex}

As we pointed out earlier, while Theorem~\ref{orderrho} is essentially a result about commuting integer matrices, our proof is indirect and makes heavy use of Perron--Frobenius theory. So, as a reality check, we tried to prove it directly. For $2$-graphs with two components, it was quite easy to see why it works (see Example~\ref{ex2dumbbell}). It was a little harder to see what was going on for three components, but the exercise was instructive --- it led us, for example, to the strengthening of Theorem~\ref{orderrho} in Corollary~\ref{orderrho2} (see Remark~\ref{r=0}).

We suppose throughout this appendix that $\Lambda$ is a dumbbell $2$-graph with three strongly connected components. Thus $\Lambda$ has three vertices $u$, $v$ and $w$, and vertex matrices of the form
\begin{equation*}
A_i=\begin{pmatrix}
m_i&q_i&r_i\\0&n_i&s_i\\0&0&p_i
\end{pmatrix}\quad\text{for $i=1,2$.}
\end{equation*}
We write $C=\{u\}$, $B=\{v\}$ and $D=\{w\}$. The hypothesis $C\Lambda^{e_i}D\not=\emptyset$ and $B\Lambda^{e_i}D\not=\emptyset$ in Theorem~\ref{orderrho} says that all the $r_i$ and $s_i$ are nonzero, and we assume until further notice that this holds. 

We now aim to prove Theorem~\ref{orderrho} directly. We may as well suppose that the $j$ in the hypothesis of Theorem~\ref{orderrho} is $j=1$ (otherwise swap colours). So we assume that
\begin{equation}\label{hypsr}
p_1=\rho(A_{D,1})>\max\{\rho(A_{C,1}),\rho(A_{B,1})\}=\max\{n_1,m_1\},
\end{equation}
and aim to prove that $p_2>n_2$ and $p_2>m_2$.

The factorisation property implies that the vertex matrices $A_i$ commute. Looking at the super-diagonal entries in $A_1A_2$ and $A_2A_1$, and rearranging, shows that $A_1A_2=A_2A_1$ if and only if
\begin{itemize}
\item[(D1)] $q_2(n_1-m_1)=q_1(n_2-m_2)$,
\item[(D2)] $s_2(p_1-n_1)=s_1(p_2-n_2)$, and 
\item[(D3)] $r_2(p_1-m_1)+q_2s_1=r_1(p_2-m_2)+q_1s_2$.
\end{itemize}
The hypothesis \eqref{hypsr} implies that $p_1>n_1$, and since $s_i\not=0$, (D2) forces $p_2>n_2$.  Thus it remains to show that $p_2>m_2$.

The first case we consider is when $q_1$ and $q_2$ are both nonzero. From (D1) we deduce that the numbers $m_i-n_i$ have the same sign. If $n_1\geq m_1$, then we have $n_2\geq m_2$ and $p_2>n_2\geq m_2$. So we assume that $m_1>n_1$ and $m_2>n_2$, and aim to prove that $p_2>m_2$.

We break up each side of (D2) using $p_i=(p_i-m_i)+m_i$, multiply the resulting equation by 
$q_2$, and apply (D1) to get
\begin{align}\label{adjustC2}
q_2s_1(p_2-m_2)+q_2s_1(m_2-n_2)&=q_2s_2(p_1-m_1)+q_2s_2(m_1-n_1)\\
&=q_2s_2(p_1-m_1)+q_1s_2(m_2-n_2).\notag
\end{align}
Next, we swap sides in (D3) and multiply by $m_2-n_2$ to get
\begin{equation}\label{adjustC3}
r_1(m_2-n_2)(p_2-m_2)+q_1s_2(m_2-n_2)=r_2(m_2-n_2)(p_1-m_1)+q_2s_1(m_2-n_2).
\end{equation}
Adding both sides of \eqref{adjustC2} and \eqref{adjustC3} gives an equation in which each of $q_1s_2(m_2-n_2)$ and $q_2s_1(m_2-n_2)$ appears on both sides; cancelling them gives
\begin{align*}
r_1(m_2-n_2)(p_2-m_2)+&q_2s_1(p_2-m_2)\\&=r_2(m_2-n_2)(p_1-m_1)+q_2s_2(p_1-m_1).
\end{align*}
Equivalently, we have
\[
\big(r_1(m_2-n_2)+q_2s_1\big)(p_2-m_2)=\big(r_2(m_2-n_2)+q_2s_2\big)(p_1-m_1).
\]
Since the coefficients of $(p_2-m_2)$ and $(p_1-m_1)$ are both positive, we deduce that $p_2-m_2$ and $p_1-m_1$ have the same sign. Thus the hypothesis $p_1>m_1$ implies that $p_2>m_2$, as required.

The second case we consider is when $q_1=0=q_2$. Then (D1) gives no information and (D3) collapses to $r_2(p_1-m_1)=r_1(p_2-m_2)$. Since $r_i\not=0$, we deduce that the numbers $p_i-m_i$ have the same same sign. Since $p_1>\max\{m_1,n_1\}$, we deduce that $p_2>m_2$, as required.

The remaining case to consider is when exactly one $q_i$ is zero. If $q_1=0$ and $q_2\neq 0$, then (D3) gives
\[
r_1(p_2-m_2)=r_2(p_1-m_1)+q_2s_1>r_2(p_1-m_1)\geq 0.
\]
Thus $p_2>m_2$. On the other hand if $q_1\neq0$ and $q_2= 0$, then (D1) reduces to $q_1(n_2-m_2)=0$, and so $n_2=m_2$. Hence, $p_2>n_2=m_2$ as required. This completes the direct proof of Theorem~\ref{orderrho} for dumbbell graphs with three components.

The next remark shows that we can relax the hypotheses of Theorem~\ref{orderrho} slightly. Graphs of this type served as motivation for our development of Corollary~\ref{orderrho2}.

\begin{rmk}\label{r=0}
Again consider a $2$-graph $\Lambda$ with three vertices, and keep the above notation. Suppose that $r_1=r_2=0$, but all $q_i$ and $s_i$ are nonzero. Notice that whilst $C\Lambda^{e_i}D=\emptyset$, we still have $C\Lambda^{\NN e_i}D\neq\emptyset$ since $C\Lambda^{e_i}B=\emptyset$ and $B\Lambda^{e_i}D=\emptyset$. As before, our goal is to show that if
\[
p_1=\rho(A_{D,1})>\max\{\rho(A_{C,1}),\rho(A_{B,1})\}=\max\{n_1,m_1\},
\]
then $p_2>n_2$ and $p_2>m_2$. 

We again derive (D1--D3). Notice that (D3) reduces to 
\begin{itemize}
\item[(D3')] $q_2s_1=q_1s_2 \Longleftrightarrow \frac{q_2}{q_1}=\frac{s_2}{s_1}$.
\end{itemize}
As before, since $p_1>n_1$ and $s_i\neq 0$, (D2) forces $p_2>n_2$, and it remains to show that $p_2>m_2$. 

First we consider the case where $n_1=m_1$. Since $q_i\neq 0$, (D1) shows that $n_2=m_2$. Thus $p_2>n_2=m_2$ as required. 

Secondly, we consider the situation where $n_1\neq m_1$. Combining (D1) and (D2) with (D3'), we get that
\[
\frac{n_2-m_2}{n_1-m_1}=\frac{q_2}{q_1}=\frac{s_2}{s_1}=\frac{p_2-n_2}{p_1-n_1}.
\]
Since $p_1>n_1$ and $p_2>n_2$, we must have that $n_2\neq m_2$, and $n_2-m_2$ and $n_1-m_1$ must have the same sign. If $n_1>m_1$, then $n_2>m_2$, and so $p_2>n_2>m_2$ as required. Alternatively, $n_1<m_1$. Since $p_1>m_1$ by assumption, we have $0<m_1-n_1<p_1-n_1$. Thus
\[
\frac{p_2-n_2}{p_1-n_1}=\frac{m_2-n_2}{m_1-n_1}>\frac{m_2-n_2}{p_1-n_1}.
\]
Hence, $p_2-n_2>m_2-n_2$, and so $p_2>m_2$. 
\end{rmk}

\begin{rmk}
Now we consider the case where $s_i=0$ for $i=1,2$. These graphs do not satisfy the hypotheses of either Theorem~\ref{orderrho} or Corollary~\ref{orderrho2}, and we shall see that they are examples of $2$-graphs in which the conclusions of Theorem~\ref{orderrho} or Corollary~\ref{orderrho2} do not hold.

When $s_i=0$, (D2) says nothing, and (D3) reduces to $r_2(p_1-m_1)=r_1(p_2-m_2)$. Thus $p_1>\max\{n_1,m_1\}$ implies $p_2>m_2$. But there is no relation relating $p_2$ to $n_2$. Indeed, consider the matrices
\[
A_1=\begin{pmatrix}
1&2&2\\0&3&0\\0&0&5
\end{pmatrix}
\quad\text{and}\quad
A_2=\begin{pmatrix}
1&3&1\\0&4&0\\0&0&3
\end{pmatrix}.
\]
Then there are $2$-graphs $\Lambda$ with these vertex matrices, and in fact many: see \cite[\S6]{KP} (or \cite{HRSW} for a concrete description of these graphs). For such $\Lambda$, we have
\begin{align*}
\rho(A_{D,1})=5&>\max\{\rho(A_{B,1}),\rho(A_{C,1})\}=\max\{3,1\}=3,\quad\text{but}\\
\rho(A_{D,2})=3&<\rho(A_{B,2})=4.
\end{align*}
So there is no version of Theorem~\ref{orderrho} for such $\Lambda$.
\end{rmk}

\end{document}